\documentclass[preprint,authoryear]{imsart} 

\RequirePackage{amsthm,amsmath,amsfonts,amssymb}
\RequirePackage[colorlinks,citecolor=blue,urlcolor=blue]{hyperref}
\RequirePackage{graphicx}

\RequirePackage[authoryear]{natbib}

\startlocaldefs
\usepackage{amsthm}

\usepackage{comment}
\usepackage{enumitem}
\usepackage{mathtools}
\usepackage{xcolor}
\usepackage{epsfig}

\startlocaldefs
\numberwithin{equation}{section}
\newtheorem{lemma}{Lemma}[section]

\newtheorem{theorem}{Theorem}[section]

\newtheorem{assumption}{Assumption}

\newtheorem{remark}{Remark}[section]
\newtheorem{example}{Example}[section]

\newcommand{\black}{\color{black}}

\endlocaldefs

\begin{document}

\begin{frontmatter}
\title{A Second-Order Extension\\ of  H\'{a}jek's  Convolution Theorem \\ with  Statistical Applications}
\runtitle{Second-Order Convolution Theorem}

\begin{aug}
\author[A]{\fnms{Junichi}~\snm{Hirukawa}\ead[label=e1]{hirukawa@nanzan-u.ac.j}}
\author[B]{\fnms{Masanobu}~\snm{Taniguchi,}\ead[label=e2]{taniguchi@waseda.jp}}
 and 
\author[C]{\fnms{Marc}~\snm{Hallin}\ead[label=e3]{mhallin@ulb.ac.be}}

\address[A]{Nanzan University, Nagoya, Japan\printead[presep={,\ }]{e1}}
\address[B]{Waseda University, Tokyo, Japan\printead[presep={,\ }]{e2}}
\address[C]{Universit\' e libre de Bruxelles, Brussels, Belgium \\ and \\ Institute of Information Theory and Automation, Czech Academy of Sciences, Prague, Czech Republic\\ \printead[presep={}]{e3}}

\end{aug}

\begin{abstract}
For a class of regular estimators, H\'{a}jek, in his celebrated  ``Convolution Theorem,'' showed that the asymptotic distribution of a regular estimator is the convolution of the distribution of an efficient estimator and some residual distribution. This result constitutes the foundation of the concept of asymptotic efficiency of regular estimators. In this paper, we provide a second-order version of that classical result. Introducing a class of second-order regular estimators with a valid Edgeworth expansion, we derive their asymptotic distribution under contiguous alternatives and show that it is the convolution of the second-order efficient distribution and some second-order residual distribution. This constitutes a second-order extension of H\'{a}jek's convolution theorem.  Based on this, we introduce a concept of {\it second-order robustness} for second-order regular estimators. For a class of general Bayes estimators and minimum contrast estimators in time series models, this second-order robustness is used (i) in the characterization of second-order robust priors, (ii) in a comparison between the second-order robustness of maximum likelihood and Whittle estimators. 
\end{abstract}


\end{frontmatter}

\section{Introduction}

{The foundational concept of Local Asymptotic Normality (LAN) and the basic results in the so-called {\it asymptotic theory 
  of statistical experiments}  originate in a series of papers by Lucien Le Cam, starting with  \cite{LeCam1960} and  culminating in the introduction of the convergence of experiments \citep{LeCam1972}. See \cite{LeCam1986}, \cite{LeCam1990} or \cite{vdVaart1998} for  definitive definitions and related properties. 
 A fundamental insight gleaned from this theory  
 is that asymptotics crucially depend on the approximation of local log-likelihood ratio processes. In particular,   LAN families are not limited to i.i.d.\ or a few simple models, but also include many dependent and time series models. 
 
Introducing a class of regular estimators under LAN conditions, \cite{Hajek1970} showed that the asymptotic distribution of a regular estimator under {\it local}  perturbations of the model parameters is the 
  convolution of the  distribution of an asymptotically efficient estimator and some residual distribution.  This result and its extension by \cite{LeCam1972} is referred to as the {\it H{\' a}jek} or  {\it H{\' a}jek-Le Cam convolution Theorem}.   
   A simple proof 
   was provided by \cite{Bickel} and can be found in \cite{Roussas1972}.  \cite{Jeganathan1982} extended the convolution theorem to the locally asymptotically {\it mixed} normal (LAMN) case and obtained the asymptotic lower bound for the risk functions of estimators (see also \cite{Jeganathan1995}).

Establishing the asymptotic distributions of test statistics under local alternatives is essential for the study of their power properties and optimality. It is well-known, however, that distinct estimators and test statistics share the same limiting behavior. Therefore, it is desirable to discriminate between them through a more detailed study of their asymptotic behavior. This gave rise to the theory commonly referred to as  higher-order asymptotics, leading to higher-order comparisons,  concepts of higher-order asymptotic efficiency,  and estimators and tests with  increased accuracy or enhanced performance.

Following \cite{Akahira1981}, higher-order asymptotic efficiency is generally defined  in terms of  probability concentration   and Edgeworth expansions.  
 The so-called class of  {\it $k$th-order asymptotically median-unbiased estimators}  ($k$th order AMU for short) was introduced in this context.  In view of the fundamental lemma of Neyman and Pearson, the {\it bound distribution} of $k$th-order AMU estimators\footnote{The {\it  
 $k$th-order bound distribution} is defined as  the distribution of the $k$th-order AMU estimators that yields the highest probability concentration around the true value.}  is derived from the Edgeworth expansion for the likelihood ratio.  For a detailed discussion of these results, see, for example, \cite{Akahira1981}.

In  i.i.d.\ settings, \cite{Fisher1925}, \cite{Rao1962}, \cite{Pfanzagl1978}, and \cite{Akahira1981} introduced higher-order asymptotic efficiency of estimators based on higher-order approximations (Edgeworth expansions) and claimed the superiority of the MLE in their setting. Second-order efficient empirical Bayes confidence intervals were constructed in \cite{Yoshimori2014}. In the-time series context, \cite{Akahira1981} considered the second-order Edgeworth expansion of the distribution of the least squares estimator for AR$\left(1\right)$ model and showed that, under Gaussian assumptions, modifying the MLE  to make it second-order AMU, the modified estimator is  second-order asymptotically efficient (see  \cite{Akahira1981}). The fact that the distributions of autocovariances from appropriate linear processes admit valid Edgeworth expansions is established by \cite{Bose1988a, Bose1988b}, and \cite{Bose1990} allows for bootstrapping the distributions of  least squares estimates in AR$\left(p\right)$ models. The third-order asymptotic theory of tests for a general stochastic process is considered in \cite{Taniguchi1991b}, who shows that the modified Rao test is Bartlett-adjustable. It is known that this type of Bartlett-type adjustment is not unique. \cite{Chandra1991} and \cite{Cordeiro1991} proposed alternative Bartlett-adjusted versions, while \cite{Rao1995} evaluated the third-order powers of the modified tests of \cite{Taniguchi1991b}, \cite{Chandra1991} and \cite{Cordeiro1991}.  \cite{Rao1997} addressed the problem of comparing the higher-order power of tests in their original forms rather than their bias-corrected or Bartlett-type adjusted versions.

A higher-order generalization of Le Cam's third lemma is obtained in \cite{Taniguchi1992} based on second-order Edgeworth approximations. That result provides an explicit form of the second-order Edgeworth expansion of the distribution of test statistics under contiguous alternatives when  their second-order asymptotic distribution   under the null hypothesis is known. However, no  higher-order generalization of the   H{\' a}jek  Theorem  
 exists in the literature. This is the motivation for this paper, where,   introducing an appropriate class of second-order regular estimators, we derive a second-order generalization of this convolution theorem where convolution   is between the second-order efficient distribution and a second-order residual distribution. As an application,   we establish  the second-order asymptotic efficiency of  second-order regular estimators.

The paper is organized as follows. Section \ref{sec:2} introduces the classical H\'{a}jek convolution theorem and some of  its consequences. In Section \ref{sec:3}, denoting by ${\rm P}_{\theta, T}$ the distribution of a stretch of $T$ observations under parameter value $\theta\in\Theta$,  we introduce a class $\mathcal{A}_\theta^{(2)}$ of sequences of second-order {\it regular} estimators $\widehat{\theta}_T$.  Then, we show that the   distributions, under~${\rm P}_{\theta_T, T}$ contiguous to ${\rm P}_{\theta, T}$, of an adequately scaled deviation $\widehat{\theta}_T - \theta$  
  converges (weakly, as $T\to\infty$) to the convolution of the second-order asymptotically efficient distribution and some second-order residual distribution that depends on the choice of  $\widehat{\theta}_T$. 
Section~\ref{sec:4} proposes a quantitative concept {\it second-order robustness}  based on an evaluation of the stability at~$\theta\in\Theta$ of the {\it second-order mean}\footnote{The  {\it second-order mean}    of a regular estimator~$\widehat{\theta}_T$ is  the second-order bias of its second-order stochastic expansion.}   of a regular estimator~$\widehat{\theta}_T$. Then, we compute that robustness for a few examples. \cite{Swe1991} introduced a class of Bayesian estimators with  prior density~$\xi$ for general ARMA$\left(p,q\right)$ models; we provide here the  prior distributions~$\xi$ for which these   estimators achieve second-order robustness.  Based on  their second-order stochastic expansions, we also  evaluate  the second-order robustness of the minimum contrast estimators proposed by \cite{Taniguchi1987}  for the parameters of general ARMA$\left(p,q\right)$ models, allowing for interesting comparisons between the maximum likelihood  and the Whittle estimators.

\section{
 H\'{a}jek's Convolution Theorem}\label{sec:2}
 \subsection{Convolution}

 Let  ${\rm P}_{\theta,T}$,  where the unknown parameter
   $\theta$   ranges over an open subset~$\Theta$   of $\mathbb R$, denote the distribution of the observed stretch 
$$\boldsymbol{X}_T\coloneqq \left(X_1,\ldots,X_T\right)^{\top},\qquad T\in\mathbb N$$  
of some   stochastic process  ${\mathcal X}\coloneqq \left\{X_t\vert\,  t\in\mathbb Z  \right\}$.  Let $\left\{c_T\vert\, \ T\in\mathbb N\right\}$ be a sequence of strictly posi\-tive real numbers such that $c_T\to\infty$ as $T\to\infty$, charac\-terizing the {\it contiguity rate} in the sequence of  families~${{\mathcal P}_T\coloneqq\left\{ {\rm P}_{\theta,T} \vert\,  \theta\in\Theta \right\}}$, $T\in{\mathbb N}$ para\-metrized by $\theta$: namely,  such that the sequences ${\rm P}_{\theta,T+c^{-1}_Th}$ and ${\rm P}_{\theta,T}$ are contiguous as $T\to\infty$ for all $\theta\in\Theta$ and~$h\in~\!\mathbb R$ (typically, $c_T={T^{1/2}}$). Considering  {\it local sequences} of para\-meter values, of the form~$\left\{ \theta_T\coloneqq \theta + c_T^{-1}h\vert\,  T\in\mathbb N \right\}$ with~$h\in{\cal H}\subset \mathbb R$, define the {\it local likelihood} and {\it log-likelihood ratios} 
\begin{align}\label{eq:2-1}L_T\left(\theta,\theta_T\right)\coloneqq \frac{{\rm dP}_{\theta_T,T}}{{\rm dP}_{\theta,T}}
\quad\text{and}\quad 
\Lambda_T\left(\theta,\theta_T\right)\coloneqq \log{L_T\left(\theta,\theta_T\right)}
\end{align}
where ${{\rm dP}_{\theta_T,T}}/{{\rm dP}_{\theta,T}}$ stands for the Radon-Nikodym derivative of the part  of~${{\rm P}_{\theta_T,T}}$~which is absolutely continuous with respect to ${{\rm P}_{\theta ,T}}. 
$

Denote by $\mathcal{L}\left(Y_T\mid {\rm P}_{\theta,T}\right)$ the distribution under   ${\rm P}_{\theta,T}$ of an ${\boldsymbol X}_T$-measurable  random variable $Y_T$, that is, the distribution of $Y_T$ when ${\boldsymbol X}_T\sim {\rm P}_{\theta,T}$,   by~$L^Y_{\theta, T}$ the corresponding distribution function. When\-ever~$\mathcal{L}\left(Y_T\mid {\rm P}_{\theta,T}\right)$ converges weakly to some probability measure~${\mathcal L}_\theta$ (notation: $\mathcal{L}\left(Y_T\mid {\rm P}_{\theta,T}\right)\stackrel{w}{\to} {\mathcal L}_\theta$   as~$T\to\infty$),  we write $Y_T\stackrel{d}{\to} Y$ for any $Y$ such that $Y\sim{\cal  L}_\theta$. 

The notation $f{\ast} g$ is used for the convolution of two integrable functions $f$ and $g$ from~$\mathbb R$ to~$\mathbb R$: namely,  
$$(f{\ast} g) (x)\coloneqq \int_{-\infty}^\infty f(x-t)g(t) {\rm d}t = \int_{-\infty}^\infty f(t)g(x-t) {\rm d}t\quad\text{for all $x\in{\mathbb R}$} .
$$
The convolution $ {\rm P}_1{\ast} {\rm P}_2\eqqcolon {\rm P}$ of two probability distributions ${\rm P}_1$ and ${\rm P}_1$ 
with distribution functions~$F_1$ and $F_2$  
 is the distribution ${\rm P}$, with  distribution function~$F$, 
  of the sum $\xi_1 + \xi_2$ of two independent variables~$\xi_1\sim{\rm P}_1$ and $\xi_2\sim{\rm P}_2$, with
$$ {\rm P}(A) =  {\rm P}_1{\ast} {\rm P}_2(A) = \int_{-\infty}^\infty {\rm 1}_A (x_1 + x_2) {\rm dP}_1{\rm dP}_2
$$
for all Borel sets $A$, hence 
$F=F_1{\ast} F_2
$
and, in case $ {\rm P}_1$ and $ {\rm P}_2$ admit Lebesgue densities $f_1$ and $f_2$ (so that $\rm P$ admits a density $f$),  $f=f_1{\ast} f_2$.

\subsection{The Convolution Theorem}
Consider  
 the class   
\begin{align}\nonumber
 {\mathcal A}^{(1)}_{{\theta}} \coloneqq 
 \Big\{
\{\widehat{\theta}_T\vert\, T\in{\mathbb N} \} \vert\, & {\text{there exists a distribution ${\mathcal L}_{\theta}$ such that, for any $h$}, }
\\
 \, & \hspace{20mm}
\mathcal{L}\big(
c_T(\widehat{\theta}_T-\theta_T) \mid {\rm P}_{\theta_T,T}
\big)
\stackrel{w}{\to}    {\mathcal L}_{\theta}   
 \text{ as $T\to\infty$}\Big\} \label{eq:2-2}
\end{align}
of sequences  $\{\widehat{\theta}_T \vert\,  T\in{\mathbb N}\}$ of estimators of $\theta$, where  ${\mathcal L}_{\theta}$ denotes some distribution that does not depend on $h$. Call $\mathcal{A}^{(1)}\coloneqq \bigcap_{\theta\in\Theta} \mathcal{A}_{{\theta}}^{(1)}$ the class of {\it {first-order } regular sequences of estimators};  this class excludes some weird phenomena as super\-efficiency. With a slight abuse of language, we also say that~$\widehat{\theta}_T$ itself is first-order regular. 

 \begin{assumption}[LAN]\label{LAN}{\rm For all $\theta\in\Theta$ and~$h~\!\in~\!{\cal H}$,
\begin{align}\label{eq:2-3}
\Lambda_T\left(\theta,\theta_T\right)=h\Delta_T\left(\theta\right)-\frac{1}{2}h^2\Gamma\left(\theta\right)+\zeta_{\theta,T}\left(h\right)
\end{align}
for some ${\bf X}_T$-measurable $\Delta_T\left(\theta\right)$ (the {\it central sequence}) and  $\Gamma(\theta)\neq 0$ (the {\it Fisher information for}~$\theta$) independent of $h$,  where, under~${\rm P}_{\theta,T}$, as $T\to\infty$, 
 $\Delta_T\left(\theta\right) $ is asymptoti\-cally~$N\left(0,\Gamma\left(\theta\right)\right)$ and $\zeta_{\theta,T}\left(h\right)$~is~$o_{\rm P}(1)$.
 }\end{assumption}

The following theorem is due to~\cite{Hajek1970}.

\begin{theorem}[H\'{a}jek's Convolution Theorem]\label{ConvThm}

Let  Assumption~\ref{LAN} hold and 
assume that~$\{\widehat{\theta}_T\vert\, T\in{\mathbb N} \}\in{\mathcal A}_{{\theta}}^{(1)}$ is such that  $\mathcal{L}\big(c_T(\widehat{\theta}_T-\theta_T)\mid {\rm P}_{\theta_T,T}\big)\stackrel{w}{\to}   {\mathcal L}_{\theta}$.  Then, ${\mathcal L}_{\theta}$ is the convolution\begin{equation}\label{conv}
{\mathcal L}_{\theta} = N\left(0,\Gamma^{-1}\left(\theta\right)\right) {\ast} {\mathcal M}_\theta .
\end{equation} 
 of a Gaussian $N\!\left(0,\Gamma^{-1}\left(\theta\right)\right)$ distribution with some other distribution~${\cal M}_\theta$. Equi\-valently, denoting by $L_\theta$, $M_\theta$, and $\Phi\left(\cdot\mid\Gamma^{-1}\left(\theta\right)\right)$, respectively,  the distribution functions of   ${\mathcal L}_{\theta}$, ${\cal M}_\theta$,  and   
  the Gaussian~$N\!\left(0,\Gamma^{-1}\left(\theta\right)\right)$,   
     $L_\theta$
  is of the form   
 \begin{align}\label{eq:2-4}
L_{\theta}\left(v\right)=\int \Phi\left(v-u\mid\Gamma^{-1}\left(\theta\right)\right){\rm d}{M}_\theta \left(u\right)\qquad v\in{\mathbb R}.
\end{align}
\end{theorem}

In other words,   $c_T(\widehat{\theta}_T-\theta_T)$ converges in distribution, as $T\to\infty$,  to the sum  of a Gaussian~$N\!\left(0,\Gamma^{-1}\left(\theta\right)\right)$ variable and some independent random variable with distribution~${\cal M}_\theta$. As a corollary of this theorem, it has been shown (see, e.g., \cite{Jeganathan1995} or \cite{Taniguchi2000}) that, for a  loss function~$\ell((\widehat{\theta} , \theta )$ of the form~$l(c_T((\widehat{\theta} - \theta))$, letting~$\Delta(\theta)\sim N(0,\Gamma(\theta))$, 
 \begin{enumerate}
\item[(i)] 
$\displaystyle{\liminf_{T\to\infty}{\rm E}_{\theta,T}\left[
l\big(
c_T(\widehat{\theta}_T-\theta)\big)\mid {\rm P}_{\theta,T}
\right]
\ge {\rm E}\left[l\left(\Gamma^{-1}(\theta)\Delta(\theta)\right)\right]
}$\vspace{1mm}
(where ${\rm E}_{\theta,T}$  de\-no\-tes expectation   under ${\rm P}_{\theta,T}$);\medskip 
\item[(ii)]  equality in (i) holds  iff
\end{enumerate}
\begin{align}\label{eq:2-6}
c_T(\widehat{\theta}_T-\theta )-\Gamma^{-1}\left(\theta\right)\Delta_T\left(\theta\right)=o_{\rm P}\left(1\right)\quad\mbox{under ${\rm P}_{\theta,T}$, as $T\to\infty$}. 
\end{align}
It follows that any $\widehat{\theta}_T$ satisfying (\ref{eq:2-6}) is asymptotically normal and {\it first-order asymptotically efficient}.

\black 

\section{A second-order 
Convolution Theorem}\label{sec:3}

In this section, we develop the second-order asymptotic theory for H\'{a}jek's Convolution Theorem~\ref{ConvThm}. 

\subsection{A third-order approximation of log-likelihood ratios}

Second-order asymptotic results are based on a third-order refinement  
 of the second-order approximation~\eqref{approxMasa}. That refinement was  obtained  by  
 \cite{Taniguchi1992}   
(see Lemma \ref{lem:Taniguchi1992} for details), and 
 requires additional  regu\-larity conditions, which we now provide. Assumptions~\ref{assum:B1-2} and~\ref{assum:B3}  incorporate 
   Assumptions~B1--B2 and   
Assumption B3, respectively, in \cite{Taniguchi1992}.

\begin{assumption} 
\label{assum:B1-2}{\rm 
Denoting by $p_{\theta,T} 
$   the probability density function of $\boldsymbol{X}_T$ under~${\rm P}_{\theta ,T}$ and writing ${\rm E}_{\theta, T}$ for the corresponding expectation,  
\begin{enumerate}
  \item[(i)]  $\theta\mapsto p_{\theta,T}\left(\boldsymbol{X}_T\right)$ is ${\rm P}_{\theta,T}$-a.s.\  four times continuously differentiable for\linebreak all~$\theta\in~\!\Theta$ and $T\in{\mathbb N}$; 
  \item[(ii)]  partial derivatives ${\partial}/{\partial \theta}$ and   expectation ${\rm E}_{\theta,T}$    
   commute for  all~$\theta\in\Theta$ \linebreak and~$T\in~\!{\mathbb N}$; 
  \item[(iii)] the   cumulants under ${\rm P}_{\theta,T}$ of
\begin{align*}
Z_i\left(\theta,T\right)\coloneqq c_T^{-1}\left[\frac{\partial^{i}}{\partial \theta^{i}}\log{p_{\theta,T}\left(\boldsymbol{X}_T\right)}-{\rm E}_{\theta ,T}\Big[\frac{\partial^{i}}{\partial \theta^{i}}\log{p_{\theta,T}\left(\boldsymbol{X}_T\right)}\Big]\right] 
\end{align*}
admit, for all $\theta\in\Theta$, asymptotic expansions of the form
\begin{align*}
&{\rm cum}_{\theta,T}\left\{Z_i\left(\theta,T\right),Z_j\left(\theta,T\right)\right\}=\kappa^{\left(1\right)}_{ij}\left(\theta\right)+O\left(c_T^{-2}\right),\\
&{\rm cum}_{\theta,T}\left\{Z_i\left(\theta,T\right),Z_j\left(\theta,T\right),Z_k\left(\theta,T\right)\right\}=c_T^{-1}\kappa^{\left(1\right)}_{ijk}\left(\theta\right)+O\left(c_T^{-2}\right)
\end{align*}
for $i,j,k=1,2,3$, and 
\begin{align*}
{\rm cum}_{\theta,T}\left\{Z_{i_1}\left(\theta,T\right),\ldots,Z_{i_J}\left(\theta,T\right)\right\}=O\left(c_T^{-J+2}\right) 
\end{align*}
for $i_1,\ldots,i_J\in\left\{1,2,3\right\}$ and $J\ge 4$, as $T\to\infty$;
  \item[(iv)] for all $\theta\in\Theta$ and $T\in{\mathbb N}$, $\displaystyle{{\rm E}_{\theta ,T}\left\{Z_3^2\left(\theta,T\right)\right\}<\infty}$ and, as $T\to\infty$,  
\begin{align*}
&c_T^{-2}{\rm E}_{\theta ,T}\left\{\frac{\partial^{2}}{\partial \theta^{2}}\log{p_{\theta,T}\left(\boldsymbol{X}_T\right)}\right\}=-\kappa^{\left(1\right)}_{11}\left(\theta\right)+O\left(c_T^{-2}\right),\\
&c_T^{-2}{\rm E}_{\theta ,T}\left\{\frac{\partial^{3}}{\partial \theta^{3}}\log{p_{\theta,T}\left(\boldsymbol{X}_T\right)}\right\}=-3\kappa^{\left(1\right)}_{12}\left(\theta\right)-\kappa^{\left(1\right)}_{111}\left(\theta\right)+O\left(c_T^{-2}\right).
\end{align*}
\end{enumerate}
}
\end{assumption}

\begin{assumption}
\label{assum:B3}
\rm{There exists $\delta>0$ such that
\begin{align*}
{\rm E}_{\theta ,T}\left[\sup_{\eta\in[\theta\pm\delta]}\left\{c_T^{-2}\frac{\partial^4}{\partial\eta^4}\log{p_{\eta,T}\left(\boldsymbol{X}_T\right)}\right\}^2\right]<\infty
\end{align*}
for every $\theta\in\Theta$ and $T\in{\mathbb N}$.
}
\end{assumption}
Then, we obtain for  $\Lambda_T\left(\theta,\theta_T\right)$ the following third-order approximation result.\smallskip

\begin{lemma}[Lemma 2 of \cite{Taniguchi1992}]\label{lem:Taniguchi1992}
Under Assumptions  \ref{assum:B1-2} and \ref{assum:B3}, the log-likelihood ratio~\eqref{eq:2-1} admits, as $T\to\infty$ under ${\rm P}_{\theta,T}$, the stochastic expansion 
\begin{align}\label{approxMasa}
\Lambda_T\left(\theta,\theta_T\right)=\lambda_{\theta,T}\left(h\right)+c_T^{-2}\zeta_{\theta,T}
\end{align}
with
\begin{align}\label{eq:a1}
\lambda_{\theta,T}\left(h\right)=h\Delta_T\left(\theta\right)+\frac{h^2}{2c_T}Z_2\left(\theta\right)-\frac{h^2}{2}\Gamma\left(\theta\right)-\frac{h^3}{6c_T}\left\{3J\left(\theta\right)+K\left(\theta\right)\right\}
\end{align}
where 
$$\Delta_T\left(\theta\right)\coloneqq Z_1(\theta,T),\ \Gamma (\theta)\coloneqq \kappa^{(1)}_{11}(\theta),\  J(\theta)\coloneqq  \kappa^{(1)}_{12}(\theta),\  M(\theta)\coloneqq  \kappa^{(1)}_{22}(\theta),$$ 
$$ K(\theta)\coloneqq  \kappa^{(1)}_{111}(\theta),\ 
N(\theta)\coloneqq  \kappa^{(1)}_{112}(\theta),$$  
and $\zeta_{\theta,T}$ satisfies, for some $\eta>0$, 
\begin{align*}
{\rm P}_{\theta,T}\left\{\left|\zeta_T\right|>\rho_Tc_T\right\}=O\left(c_T^{-2+\eta}\right)  
\end{align*}
for any $\theta\in\Theta$ and $\rho_T$ such that $\rho_T\to 0$ and $\rho_Tc_T\to\infty$ as $T\to\infty$.
\end{lemma}

\subsection{Second-order regularity}

Second-order results also require reinforcing the regula\-rity assumptions on the estimators ${\widehat{\theta}}_T$. \cite{Taniguchi1983} (see also   \cite{Taniguchi2000}) therefore introduces the  concept 
 of   {\it asymptotically orthogonal variables}.  

Let Assumption~\ref{LAN} hold. Define 
\begin{equation}\label{W1}
W_{1,T}\left(\theta\right)\coloneqq \Gamma^{-1/2}\left(\theta\right)\Delta_T\left(\theta\right)\
\end{equation}
and
\begin{equation}\label{W2}
W_{2,T}\left(\theta\right)\coloneqq \sigma_2^{-1}\left\{Z_2\left(\theta,T\right)-J\left(\theta\right)\Gamma^{-1}\left(\theta\right)\Delta_T\left(\theta\right)\right\}
\end{equation}
with
 $
\sigma_2^2\coloneqq M\left(\theta\right)-J^2\left(\theta\right)\Gamma^{-1}\left(\theta\right)$.
Then, the variables $W_{1,T}\left(\theta\right)$ and $W_{2,T}\left(\theta\right)$ are called {\em {asympto\-tically} orthogonal} if they satisfy, as {$T\to\infty$,}
\begin{align*}
&{\rm cum}_{\theta}\left(W_{1,T}\left(\theta\right),W_{1,T}\left(\theta\right)\right)=1+O\left(c_T^{-2}\right),\\
&{\rm cum}_{\theta}\left(W_{2,T}\left(\theta\right),W_{2,T}\left(\theta\right)\right)=1+O\left(c_T^{-2}\right),
\end{align*}
and
\begin{align*}
&{\,}\hspace{-5.9mm}{\rm cum}_{\theta}\left(W_{1,T}\left(\theta\right),W_{2,T}\left(\theta\right)\right)=O\left(c_T^{-2}\right).
\end{align*}

We can  
rewrite \eqref{approxMasa} as
\begin{align}\nonumber
\lambda_{\theta,T}\left(h\right)=h\Gamma^{1/2}\!\left(\theta\right)W_{1,T}\!\left(\theta\right)&+\frac{h^2}{2c_T}\left\{J\left(\theta\right)\Gamma^{-1/2}\!\left(\theta\right)W_{1,T}\!\left(\theta\right)+\sigma_2W_{2,T}\!\left(\theta\right)\right\}\\
&-\frac{h^2}{2}\Gamma\!\left(\theta\right)-\frac{h^3}{6c_T}\left\{3J\!\left(\theta\right)+K\!\left(\theta\right)\right\}.
\label{eq:a2-2}
\end{align}

Denoting by $\phi$ the standard normal density and by~$H_j$ the $j$th Hermite polynomial, let
$${\rm Edg}^{\left(2\right)}_\theta\!\coloneqq \Phi_{1,\theta} +c_T^{-1} \phi_{2,\theta}$$ where $\Phi_{1,\theta}$  stands for some Gaussian distribution function and~{$\phi_{2,\theta}$ is some {\it Hermite  density} of the form\footnote{Note that ${\rm Edg}^{\left(2\right)}_\theta $ is not necessarily a distribution function and even can take negative values. }
$$ \phi_{2,\theta}: x\mapsto  \phi_{2,\theta}(x)\coloneqq \phi(x) [c_{1,\theta}H_1(x) + c_{2,\theta} H_3(x)]$$ for some  $c_{1,\theta}$  and $c_{2,\theta}$  depending on $\theta$.} Denote by~$L^S_{ \theta_T,T}$ the distribution function, under~${\rm P}_{\theta_T, T}$ (with~$\theta_T=\theta+c_T^{-1}h$), of $S_{\theta_T,T}\!\coloneqq c_T(\widehat{\theta}_T\!-~\!\theta_T )$. 
We say that~${\rm Edg}^{\left(2\right)}_\theta$ constitutes a {\it valid} second-order Edgeworth approximation  for~$L^S_{{\theta}_T,T}$~iff %
\begin{align}\label{eq:3-6}
\sup_{x\in {\mathbb R}}\left\vert L^S_{{\theta}_T,T}(x) - {\rm Edg}^{(2)}_{{\theta_T}} (x)\right\vert = o(c_T^{-1})\quad\text{as $T\to\infty$}.%
\end{align}

 Next consider the class  
\begin{align}\label{eq:3-9}
\mathcal{A}_\theta^{(2)}\coloneqq
\Big\{
\{\widehat{\theta}_T\vert\, T\in{\mathbb N}\} 
\Big\vert\, &\text{there exists $\Phi_{1,\theta}$ and $\phi_{2,\theta}$}  \text{ such that, for all $h$, ${\rm Edg}^{\left(2\right)}_\theta$  is}\nonumber \\ 
&\hspace{10mm}\text{  a valid second-order Edgeworth approximation  for~$L^S_{\theta_T,T}$}
\Big\} 
\end{align}
of sequences  $\{\widehat{\theta}_T \vert\, T\in{\mathbb N}\}$ of estimators of $\theta$ for which $L^S_{\theta_T,T}$  admits a valid second-order Edgeworth approximation of the form ${\rm Edg}^{(2)}_\theta$ that does not depend  on $h$:  $\mathcal{A}_\theta^{(2)}$, thus, is a second-order restriction of ${\mathcal{A}}_\theta^{(1)}$. 
Call $$\mathcal{A}^{(2)}
\coloneqq
\bigcap_{\theta\in\Theta}
{\mathcal A}_\theta^{(2)}\subset {\mathcal{A}}_\theta^{(1)}$$ 
 the class of {\it second-order regular sequences of estimators}  of $\theta$.  With a slight abuse of language, we also say that $\widehat{\theta}_T$ itself is second-order regular.\smallskip

Our second-order  H\'{a}jek convolution result holds for sequences of estimators  in the intersection of~$\mathcal{A}^{(2)}$ with the class $\mathcal{D}\coloneqq \bigcap_{\theta\in\Theta}\mathcal{D}_\theta$, where 
\begin{align}\label{eq:3-10}
\mathcal{D}_\theta
\coloneqq\Big\{
\{\widehat{\theta}_T\vert\, T\in{\mathbb N}\} 
\Big\vert\, 
& S_{\theta,T}=a_1W_{1,T}(\theta)+c_T^{-1}Q_{{\theta,}T}+o_{\rm P}\left(c_T^{-1}\right) \nonumber
\\[-1mm]
 &\hspace{5mm}\text{for some $a_1$ and some $Q_{{\theta,}T}$ under ${\rm P}_{\theta, T}$, as $T\to\infty$ }
\Big\}\end{align}
 with   $S_{\theta,T}\coloneqq c_T\left(\widehat{\theta}_T-\theta\right)$, $W_{1,T}\left(\theta\right)$ and $W_{2,T}\left(\theta\right)$ defined in \eqref{W1} and \eqref{W2}, respectively, and
 $$Q_{{\theta,}T}  \coloneqq Q_{{\theta,}T}\left(W_{1,T}\left(\theta\right),W_{2,T}\left(\theta\right)\right)=O_{\rm P}(1)\quad\text{under ${\rm P}_{\theta, T}$, as $T\to\infty$}.$$ 
This class $\mathcal{D}$ is a natural one and comprises, 
 in particular,  the maximum likelihood and  Bayesian estimators of~$\theta$  (see Section~4 below and, e.g.,~\cite{Taniguchi2000}).

\color{black}

\begin{remark}
{\rm 
Since $W_{1,T}\left(\theta\right)=\Gamma^{-1/2}\left(\theta\right)\Delta_T\left(\theta\right)$, letting  $a_1=\Gamma^{-1/2}\left(\theta\right)$ in \eqref{eq:3-10} cha\-racterizes a first-order asymptotically efficient estimator ${\widehat\theta}_T$ which, furthermore,   also  enjoys  first-order asymptotic median unbiasedness (AMU).\footnote{A sequence of estimators $\widehat{\theta}_T$ is called {\it first-order asymptotically median unbiased} (AMU, see~\cite{Akahira1981}) iff~${\rm P}_{\theta, T} [ c_T ( (\widehat{\theta}_T - \theta ) \leq 0 ]  =1/2 +  o(1)$ as $T\to\infty$.}
}
\end{remark}

\subsection{Main result}\label{sec33}
We now can state the main result of this paper. Let $H_j
$, $j\in{\mathbb N}$  denote the $j$th Hermite polynomial. 

\begin{theorem}[A second-order H\'{a}jek Convolution Theorem]\label{thm:2}
Let  Assumptions~1--3   hold. Assume furthermore that~$\{\widehat{\theta}_T\vert\, T\in{\mathbb N}\}\in {\mathcal{A}^{(2)} \cap} \mathcal{D}$ and {let $L^S_{ \theta_T,T}$ be the distribution function under~${\rm P}_{\theta_T, T}$ of~$S_{\theta_T,T}\!\coloneqq c_T(\widehat{\theta}_T\!-~\!\theta_T )$ with $\theta_T=\theta+c_T^{-1}h$.}

Then, $L^S_{ \theta_T,T}$ admits, as $T\to\infty$, the asymptotic expansion
\begin{equation}\label{Lresult}
 L^S_{\theta_T,  \theta_T}=\left[\Phi_{\Gamma(\theta)} \cdot F^{(2)}_{\theta, T}
\right]\ast G_2+O\left(c_T^{-2}\right),
\end{equation}
where $\Phi_{\Gamma(\theta)}$ stands for the $N(0,{\Gamma^{-1}(\theta))}$ distribution function, 
\begin{align*}
 F^{(2)}_{\theta ,T}  
\coloneqq 1+c_T^{-1}\Big\{{\omega^{(1)}_{\theta}}H_1
+\frac{3J\left(\theta\right)+K\left(\theta\right)}{6\Gamma^{3}\left(\theta\right)}H_3
\Big\}
\end{align*}
 with {$\omega^{(1)}_{\theta}\coloneqq {\rm E}_{\theta,T}\left(Q_{\theta,T}\right)+O\left(c_T^{-1}\right)$},  
and  $G_2$ is the distribution function,  under~$ P_{\theta,T}$, of the \em{second-order residual }  
$$
 S_{\theta,T}-\Gamma^{-1/2}\left(\theta\right)W_{1,T}\left(\theta\right)-c_T^{-1}Q_{\theta,T}
$$ 
where $W_{1,T}(\theta)$ and~$Q_{\theta,T}$ are as  in  
 \eqref{eq:3-10}.
\end{theorem}

Note that the bias-adjusted $F_{2,\theta,T}\coloneqq \Phi_{\Gamma(\theta)}{\cdot} F^{(2)}_{\theta, T}$ in \eqref{Lresult} 
 is asymptotically second-order efficient;  
  therefore,   
  any second-order bias-adjusted $\{\widehat{\theta}_T\vert\, T\in{\mathbb N}\}$ in~${\mathcal{A}^{(2)}\cap \mathcal{D}}$ 
 is asymptotically second-order efficient (see the discussion in  \cite{Taniguchi1991b}, pp.~45-49).  
 
{
The proof of  Theorem \ref{thm:2} relies on  a lemma providing the behavior of the characteristic function of~$S_{\theta_T,T}$ under local alternatives ${\rm P}_{\theta_T,T}$.  

Let  
   $\{\widehat{\theta}_T\vert\, T\in~\!{\mathbb N}\}  
\in~\!\mathcal{D}_\theta$ and recall that~$S_{\theta,T}\coloneqq\!c_T(\widehat{\theta}_T-\theta )$. 
The characteristic function of~$S_{\theta,T}$ under ${\rm P}_{\theta_T,T}$ with $\theta_T =~\!\theta +~\!c_T^{-1}h$~is 
\begin{align*}
\psi_{S_{\theta,T}}\left(t,h\right)&\coloneqq {\rm E}_{\theta+c_T^{-1}h ,T}\left\{\exp\left(it\, S_{\theta,T}\right)\right\}=\int \exp\left\{it\, S\left(\boldsymbol{x}\right)\right\}L_T\left(\boldsymbol{x}\right){\rm dP}_{\theta,T}\left(\boldsymbol{x}\right)\\
&={\rm E}_{\theta ,T}\left\{\exp\left(it\, S_{\theta,T}
+\Lambda_T\right)\right\}.
\end{align*}
Similarly, the characteristic function under ${\rm P}_{\theta_T,T}$ of $$
S_{\theta_T, T}\coloneqq c_T(\widehat{\theta}_T-\theta_T)=c_T(\widehat{\theta}_T-\theta )-h=S_{\theta,T}-h$$
 is
 {
 \begin{align}\label{eq:3-8}
 \psi_{S_{\theta_T, T}}\left(t,h\right)=\exp ({-it\,h})\psi_{S_{\theta,T}}\left(t,h\right).     
 \end{align}
 }

\begin{lemma}\label{lem:3-2}

 Let  Assumptions~1--3   hold. Assume that $\{\widehat{\theta}_T\vert\, T\in{\mathbb N}\}\in \mathcal{D}$. 
The characte\-ristic function under ${\rm P}_{\theta_T,T}$ of $S_{\theta_T, T}$ satisfies 

\begin{align}\nonumber
\psi_{S_{\theta_T, T}}\left(t,h\right)&=\exp{\Big\{\frac{1}{2}{\left(it\right)^2a_1^2}{2}+it\big\{a_1\Gamma^{1/2}\left(\theta\right)-1\big\}h\Big\}}
 \nonumber \\
 & \nonumber  
\qquad\times \left[1+c_T^{-1}\left\{it\,a_2\left(\mu\right)
 +it\,a_1\frac{h^2J\left(\theta\right)\Gamma^{-1/2}\left(\theta\right)}{2}-\frac{h^3K\left(\theta\right)}{6}\right.\right.   \\ 
\qquad &\qquad  \qquad\qquad\qquad\quad  
\left.\left. +\frac{K\left(\theta\right)\Gamma^{-3/2}\left(\theta\right)}{6}\mu^3\right\}\right]+O\left(c_T^{-2}\right)\ \text{as $T\to\infty$,}\label{eq:a4}
\end{align}
where $\mu\!\coloneqq\! it\,a_1+h\Gamma^{1/2}\!\left(\theta\right)$ and 
$ 
a_2\left(\mu\right) \!\coloneqq\! {\rm E}_{{\theta,T}}\left\{Q_{\theta,T}\!\left(W_{1,T}\!\left(\theta\right)+\mu,W_{2,T}\!\left(\theta\right)\right)\right\}+~\!o\left(1\right).  
$
\end{lemma}

\begin{proof}[Proof of Lemma \ref{lem:3-2}] 
See the Appendix. 
\end{proof}

\begin{proof}[Proof of Theorem \ref{thm:2}]

In (\ref{eq:a4}), taking $a_1 = a_{1,\theta}\coloneqq \Gamma^{-1/2}\left(\theta\right)$ so that the condition for  first-order asymptotic efficiency and/or  first-order AMU is satisfied, we obtain (denoting by $\psi_{S_{\theta_T, T}}^*$ the characteristic function under ${\rm P}_{\theta_T,T}$ of a~$S_{\theta_T,T}$ that satisfies the first-order asymptotic efficiency and/or first-order AMU condition)
\begin{align}\label{eq:chaS}
\psi_{S_{\theta_{T}, T}}^*\left(t,h\right)&=\exp{\left\{{\left(it\right)^2\Gamma^{-1}\left(\theta\right)}/{2}\right\}}\Big[1+c_T^{-1}\Big\{it\, a_2\left(\mu\right)
 \nonumber\\ 
&\qquad 
+it\frac{h^2J\left(\theta\right)\Gamma^{-1}\left(\theta\right)}{2}-\frac{h^3K\left(\theta\right)}{6}+\frac{K\left(\theta\right)\Gamma^{-3/2}\left(\theta\right)}{6}\mu^3\Big\}\Big]+O\left(c_T^{-2}\right)
\end{align}
{under ${\rm P}_{\theta_T,T}$, as $T\to\infty$}, with $\mu=it\,\Gamma^{-1/2}\!\left(\theta\right)+h\Gamma^{1/2}\!\left(\theta\right)$.  
%
Replacing $h$ with~$-it\,\Gamma^{-1}\!\left(\theta\right)$ in~\eqref{eq:chaS} yields~$\mu=0$ and 
\begin{align}\label{eq:chaS2}
{
\psi_{S_{\theta_T, T}}^{*}\left(t,-it\,\Gamma^{-1}\left(\theta\right)\right)
} 
&=\exp{\left\{{\left(it\right)^2\Gamma^{-1}\left(\theta\right)}/{2}\right\}} \nonumber \\ 
&\qquad \times \left[1+c_T^{-1}\left\{it \,\omega_1-it^3\frac{\left\{3J\left(\theta\right)+K\left(\theta\right)\right\}}{6\Gamma^{3}\left(\theta\right)}\right\}\right]+O\left(c_T^{-2}\right).
\end{align}

Let us now concentrate on the derivation of the asymptotic expansion of the cha\-racteristic function~$\psi_{S_{\theta_T, T}}$ of $S_{\theta_T, T}$ under  ${\rm P}_{\theta_T,T}$, for a sequence $\{\widehat{\theta}_T\vert\, T\in{\mathbb N}\}$  
in  the class  $\mathcal{A}_\theta^{(2)}$ of   second-order regular estimators.  
 Estimators in this second-order regular class are such that the second-order asymptotic expansion of their characteristic function   
  does not depend on $h$. 

Therefore, 
{
from (\ref{eq:3-8}) 
}
\begin{align}\label{eq:chaS2-2}
{
\psi_{S_{\theta_T, T}}^{*}\left(t,-it\,\Gamma^{-1}\left(\theta\right)\right)=\psi_{S_{\theta_T, T}}^{*}\left(t,0\right)
=\psi_{S_{\theta, T}}^{*}\left(t,0\right)
}
\eqqcolon \psi_{\Gamma}\left(t\right),
\end{align}
which is the second-order expansion of the characteristic function  under ${\rm P}_{\theta,T}$ of
\[
\Gamma^{-1/2}\left(\theta\right)W_{1,T}\left(\theta\right)+c_T^{-1}Q_{\theta, T}\left(W_{1,T}\left(\theta\right),W_{2,T}\left(\theta\right)\right).
\]

In line with  \cite{Roussas1972} (see also p.69 of  \cite{Taniguchi2000}), similarly replacing~$h$ with~$-it\Gamma^{-1}\left(\theta\right)$ in 
 \eqref{eq:a4} 
yields  $\mu =\widetilde{\mu}\coloneqq it\,\left\{a_1-\Gamma^{-1/2}\left(\theta\right)\right\}\eqqcolon it\,a$, where~$a\coloneqq a_1-\Gamma^{-1/2}\left(\theta\right)$, and
\begin{align*}
&\psi_{{S}_{\theta_T,T}}\left(t,-it\Gamma^{-1}\left(\theta\right)\right)=\exp{\left\{\frac{\left(it\right)^2a_1^2}{2}-\left(it\right)^2\left\{a_1\Gamma^{1/2}\left(\theta\right)-1\right\}\Gamma^{-1}\left(\theta\right)\right\}}\\ 
&\quad \times \left[ 1+c_T^{-1}\left\{it\,a_2\left(\widetilde{\mu}\right) - it^3\left\{a_1\frac{J\left(\theta\right)}{2\Gamma^{5/2}\left(\theta\right)}
+\frac{K\left(\theta\right)}{6\Gamma^{3}\left(\theta\right)}\right\}+\frac{K\left(\theta\right)\Gamma^{-3/2}\left(\theta\right)}{6}\widetilde{\mu}^3\right\}\right]+O\left(c_T^{-2}\right)\\
&\qquad =\exp{\left\{\frac{\left(it\right)^2}{2}\left[\left\{a_1-\Gamma^{-1/2}\left(\theta\right)\right\}^2+\Gamma^{-1}\left(\theta\right)\right]\right\}} \\ 
&\quad \times \left[1+c_T^{-1}\left\{it\,a_2\left(\widetilde{\mu}\right) 
-it^3\left\{a_1\frac{J\left(\theta\right)}{2\Gamma^{5/2}\left(\theta\right)}+\frac{K\left(\theta\right)}{6\Gamma^{3}\left(\theta\right)}\right\}+\frac{K\left(\theta\right)\Gamma^{-3/2}\left(\theta\right)}{6}\widetilde{\mu}^3\right\}\right]+O\left(c_T^{-2}\right)\\
&\qquad =\exp{\left\{\frac{\left(it\right)^2\Gamma^{-1}\left(\theta\right)}{2}\right\}}\exp{\left\{\frac{\left(it\right)^2a^2}{2}\right\}}\left[1+c_T^{-1}\left\{it\,\omega_1-it^3\frac{\left\{3J\left(\theta\right)+K\left(\theta\right)\right\}}{6\Gamma^{3}\left(\theta\right)}\right\}\right]\\
&\quad \times \left[1+c_T^{-1}\left\{it\left\{a_2\left(\widetilde{\mu}\right)-\omega_1\right\}-it^3\frac{aJ\left(\theta\right)}{2\Gamma^{5/2}\left(\theta\right)}+\frac{K\left(\theta\right)\Gamma^{-3/2}\left(\theta\right)}{6}\widetilde{\mu}^3\right\}\right]+O\left(c_T^{-2}\right),
\end{align*}
the limit as $T\to\infty$ of which does not depend on $h$. Therefore, we can conclude that 
\begin{align*}
\psi_{{S}_{\theta_T,T}}\left(t,h\right)=\psi_{{S}_{\theta_T,T}}\left(t, 0\right)+O\left(c_T^{-2}\right)
\eqqcolon \psi_{\Gamma}\left(t\right)  \psi_{G}\left(t\right)+O\left(c_T^{-2}\right)
\end{align*}
where 
\begin{align}
&\psi_{\Gamma}\left(t\right)\coloneqq \exp{\left\{\frac{\left(it\right)^2\Gamma^{-1}\left(\theta\right)}{2}\right\}}\left[1+c_T^{-1}\left\{it\omega_1-it^3\frac{\left\{3J\left(\theta\right)+K\left(\theta\right)\right\}}{6\Gamma^{3}\left(\theta\right)}\right\}\right]
\end{align}
and 
\begin{align}
\psi_{G}\left(t\right)\coloneqq& \exp{\left\{{\left(it\right)^2a^2}/{2}\right\}}  \nonumber \\
&\quad \times\left[1+c_T^{-1}\left\{it\left\{a_2\left(\widetilde{\mu}\right)-\omega_1\right\}-it^3\frac{aJ\left(\theta\right)}{2\Gamma^{5/2}\left(\theta\right)}+\frac{K\left(\theta\right)\Gamma^{-3/2}\left(\theta\right)}{6}\widetilde{\mu}^3\right\}\right].
\end{align}
From (\ref{eq:chaS2-2}), we see that $\psi_{\Gamma}\left(t\right)$ is the second-order expansion of the characteristic function of~$\Gamma^{-1/2}\left(\theta\right)W_{1,T}\left(\theta\right)+c_T^{-1}Q_{\theta,T}\left(W_{1,T}\left(\theta\right),W_{2,T}\left(\theta\right)\right)$ under ${\rm P}_{\theta,T}$, while $\psi_{G}\left(t\right)$ is the second-order expansion of the characteristic function of the residual
\begin{align*}
G&\coloneqq S_{\theta,T}-\Gamma^{-1/2}\left(\theta\right)W_{1,T}\left(\theta\right)-c_T^{-1}Q_{\theta,T}\left(W_{1,T}\left(\theta\right),W_{2,T}\left(\theta\right)\right)\\
&=\left\{a_1-\Gamma^{-1/2}\left(\theta\right)\right\}W_{1,T}\left(\theta\right)+O\left(c_T^{-2}\right)=aW_{1,T}\left(\theta\right)+{ O_{\rm P}\left(c_T^{-2}\right)}. 
\end{align*}
Note that, when $a_1\coloneqq \Gamma^{-1/2}\left(\theta\right)$ (that is, $a=0,\  \widetilde{\mu}=0,\ a_2\left(\widetilde{\mu}\right)=a_2\left(0\right)=\omega_1$),  the conditions for  first-order asymptotic efficiency and/or  first-order AMU are satisfied and 
 the residual $G$  degenerates up to the second-order. 
\end{proof}

\begin{remark}{\rm 
The statement ``first-order efficiency implies second-order efficiency'' is often used in the literature on  higher-order asymptotics,  implying that
if we make the second-order bias adjustment (second-order unbiased, second-order AMU, etc.), then the bias-adjusted estimator becomes second-order
asymptotically efficient. In fact, this proof demonstrates that the condition for  first-order asymptotic efficiency---namely, the degeneracy of the first-order residual distribution---automatically entails the degeneracy of the second-order residual distribution as well. This implies that, under the condition of   first-order asymptotic efficiency, the second-order distributions of statistics belonging to the class of second-order regular estimators coincide up to the bias-correction term. }
\end{remark}

\section{Statistical applications: second-order robustness of second-order regular estimators} \label{sec:4}
The previous sections, for notational simplicity,\footnote{See \cite{Taniguchi1994} for some fundamental results in higher-order asymptotics in the multiparameter case.} are limited to one-parameter models.  
Our approach, however, extends to the multiparameter case, with a~$K$-dimensional
 parameter of the form~${\boldsymbol\theta} =(\theta_1,\ldots,\theta_K) = (\theta_i,{\boldsymbol\theta}^i)$ where a component of interest~$\theta_i$ of $ {\boldsymbol\theta} $ 
 is singled out while~${\boldsymbol\theta}^i$ collects the~$(K-1)$ remaining ones. Specifying~${\boldsymbol\theta}^i$ yields a one-parameter submodel with parameter~$\theta_i$. If the original model is LAN with $K$-dimensional central se\-quence~${\boldsymbol\Delta}_T(\theta_i,{\boldsymbol\theta}^i)$ and information matrix~${\boldsymbol\Gamma}(\theta_i,{\boldsymbol\theta}^i)$, then the specified-${\boldsymbol\theta}^i$ submodel is LAN with (scalar) central sequence the $i$th component~$\Delta_{T,i}(\theta_i,{\boldsymbol\theta}^i)$ of~${\boldsymbol\Delta}_T(\theta_i,{\boldsymbol\theta}^i)$ and (scalar) information quantity the $i$th diagonal element $\Gamma_{ii}(\theta_i,{\boldsymbol\theta}^i)$ of~${\boldsymbol\Gamma}(\theta_i,{\boldsymbol\theta}^i)$. Our results then apply for~$\theta = \theta_i$, with $\Delta_T(\theta)\coloneqq \Delta_{T,i}(\theta_i,{\boldsymbol\theta}^i))$, and $\Gamma (\theta)\coloneqq \Gamma_{ii}(\theta_i,{\boldsymbol\theta}^i)$.  This is, under general regularity conditions, the case of the ARMA models considered below.\footnote{We refer to \cite{Jeganathan1982}, \cite{Kreiss87},  \cite{GH95},  or Chapter~2 in \cite{Taniguchi2000} for  LAN results for ARMA  and VARMA models. }  To facilitate the application of the results of Section \ref{sec:3}, the dependence on ${\boldsymbol\theta}^i$ of  $\Delta_T(\theta)$, $\Gamma (\theta)$, and other $\theta$-related quantities such as the range $\Theta$ of $\theta$ is not reflected, unless necessary,  in the notation.

\color{black} 

\subsection{Second-order robustness of estimators}\label{multiparamsec}

Let  $Q_{{\theta,}T} $ 
be  as  in the expansion of $S_{\theta,T}$ for~$\{\widehat{\theta}_T\vert\, T\in{\mathbb N}\} \in \mathcal{D}_\theta$, see \eqref{eq:3-10}.   
 Define  the second-order robustness   of $Q_{\theta, T}$ as the function  
 {
\begin{align*}
{\theta \mapsto }   R_{\, Q}\left(\theta\right) \coloneqq  \lim_{T\to\infty}  \frac{\partial }{\partial \theta} {\rm E}_{\theta{,T}} \Big[Q_{\theta,T} \left\{ W_{1,T}\left(\theta\right),W_{2,T}\left(\theta\right)\right\}\Big]
\end{align*}
 provided that the derivative and the limit exist fr $\theta\in\Theta$. 

In this section, we study this second-order robustness  for Bayesian and minimum contrast estimators, and provide some concrete practical examples in the context of ARMA models.

\subsection{Second-order robustness of Bayesian estimators for ARMA processes}
 Let $\left\{X_t\vert\, t\in{\mathbb Z}\right\}$ denote an ARMA$\left(p,q\right)$ process with spectral density  
\begin{align}\label{spectralARMA}
f_{\boldsymbol\theta}\left(\lambda\right)=\frac{\sigma^2}{2\pi}\frac{\prod_{k=1}^{q}\left(1-\psi_ke^{i\lambda}\right)\left(1-\psi_ke^{-i\lambda}\right)}{\prod_{k=1}^{p}\left(1-\rho_ke^{i\lambda}\right)\left(1-\rho_ke^{-i\lambda}\right)},
\end{align}
parametrized by ${\boldsymbol\theta}\!\coloneqq\! ( {\rho_1,\ldots,\rho_p,\psi_1,\ldots,\psi_q},\sigma^2 )\in{\boldsymbol\Theta}$ with $\left|\rho_k\right|<1$, $k=1,\ldots,p$, $\left|\psi_k\right|<1$, $k=~\!1,\ldots,q$,  and $0<\sigma^2<\infty$. Single out one component $\theta\in\Theta$       and specify all other ones. 
 Let $\xi $ be a prior density over~$\Theta$. Suppose that $\theta\mapsto \xi\left(\theta\right)$ is twice continuously differentiable on $\Theta$, with derivatives~$\xi^\prime (\theta)$ and~$\xi ^{\prime\prime}(\theta)$. The Bayes estimator of $\theta$ is defined as 
\begin{align}
\widehat{\theta}_{B,   T}\coloneqq 
{\int_{\Theta}\theta {\rm P}_{\theta,T}\left(\boldsymbol{X}_T\right)\xi\left(\theta\right){\rm d}\theta}/{\int_{\Theta} {\rm P}_{\theta,T}\left(\boldsymbol{X}_T\right)\xi\left(\theta\right){\rm d}\theta}.
\end{align}
  Putting ${z}\coloneqq \sqrt{T}\left(\theta-\theta_0\right)$, 
     we obtain
\begin{align}
\sqrt{T}\left(\widehat{\theta}_{B, T}-\theta_0\right)=\frac{{\displaystyle\int} {z}\exp{\left[\log{\left\{{\rm P}_{\theta_0+{z}/\sqrt{T}}\left(\boldsymbol{X}_T\right)\right\}}\right]}\xi\left(\theta_0+{z}/\sqrt{T}\right){\rm d}{z}}{{\displaystyle\int} \exp{\left[\log{\left\{{\rm P}_{\theta_0+{z}/\sqrt{T}}\left(\boldsymbol{X}_T\right)\right\}}\right]}\xi\left(\theta_0+{z}/\sqrt{T}\right){\rm d}{z}}.
\end{align}
Expanding ${z\mapsto}\,\log{\big\{{\rm P}_{\theta_0+{z}/\sqrt{T}}\left(\boldsymbol{X}_T\right)\big\}}$ and ${z\mapsto} \,\xi\big(\theta_0+{z}/\sqrt{T}\big)$ in a  Taylor series {at $\theta=\theta_0$ (that is, around~$z=0$)}, we obtain (see  \cite{Swe1991} and Assumption~\ref{assum:B1-2})  the stochastic expansion 
\begin{align}
\sqrt{T}\left(\widehat{\theta}_{B, T}-\theta_0\right)=\frac{Z_1(\theta, T)}{\Gamma(\theta)}+\frac{{Q_{\theta , T}}}{\sqrt{T}}+o_{\rm P}\left(1\right),
\end{align}
where 
\[{Q_{\theta , T}}=\frac{Z_1(\theta, T)Z_2(\theta, T)}{\Gamma(\theta)^2}+\frac{-3J-K}{2\Gamma(\theta)^3}Z_1^2(\theta, T)+\frac{-3J-K}{2\Gamma^2(\theta)}+\frac{\xi^{\prime} (\theta)}{\xi  (\theta)\Gamma(\theta)}
.\]

Then ${\rm E}_{\theta,T}\left({ Q_{\theta , T}}\right)= {(-2J-K)}/{\Gamma^2(\theta)}+\left(\log{\xi}(\theta)\right)^{\prime}/{\Gamma (\theta)}$, hence
\begin{align*}
R_Q(\theta)=\left[\frac{-2J-K}{\Gamma^2(\theta)}+\frac{\left(\log{\xi}(\theta)\right)^{\prime}}{\Gamma(\theta)}\right]^{\prime}. 
\end{align*}

Call {\it second-order robust prior} any prior $\xi=\xi(\theta)$  such that $\frac{\partial}{\partial \theta} R_Q(\theta) = 0$ for~$\theta\in\Theta$. 
We then have the following result.

\begin{theorem}
A prior density $\xi=\xi\left(\theta\right)$  such that $\log{\xi}(\theta)$ satisfies, for $\theta\in\Theta$, 
\begin{align}\label{eq:4-5}
\left(\frac{-2J-K}{\Gamma^2(\theta)}\right)^{\prime}+\frac{\left(\log{\xi (\theta)}\right)^{\prime\prime}}{\Gamma(\theta)}+\left(\log{\xi (\theta)}\right)^{\prime}\left(\frac{1}{\Gamma (\theta)}\right)^{\prime}=0
\end{align}
 is   second-order robust. 
\end{theorem}

{\black 
\begin{example}[ARMA$\left(p,q\right)$ case with $\theta=\rho_k$ for some $1\leq k\leq p$]
\label{ex:4-1}
 {\rm 
For the ARMA$\left(p,q\right)$ case with spectral density (\ref{spectralARMA}), assume that $\theta$ is the AR  para\-meter~$\rho_k$, $1\leq k\leq p$ while all other parameters are specified. Then  
\begin{align*}
&\Gamma\left(\theta\right)=\frac{1}{1-\theta^2},\quad K\left(\theta\right)=\frac{6\theta}{\left(1-\theta^2\right)^2},\quad \text{and} \quad J\left(\theta\right)=\frac{-2\theta}{\left(1-\theta^2\right)^2}
\end{align*}
(see p.~31 in  \cite{Taniguchi1991b}); condition (\ref{eq:4-5}) yields 
\begin{align*}
(\log{\xi(\theta)})^{\prime\prime}-\frac{2\theta}{1-\theta^2}(\log{\xi(\theta)})^{\prime}-\frac{2}{1-\theta^2}=0,
\end{align*}
leading to
\begin{align*}
(\log{\xi(\theta)})^{\prime}=\exp\Big[{\displaystyle \int\frac{2\theta}{1-\theta^2}{\rm d}\theta}\Big]\left\{\int \frac{2}{1-\theta^2}\exp\Big[{\int\frac{-2\theta}{1-\theta^2}{\rm d}\theta}\Big]{\rm d}\theta +c\right\}.
\end{align*}
Noting that  
\begin{align*}
&\int \frac{2\theta}{1-\theta^2}{\rm d}\theta =-\int\left(\frac{1}{1+\theta}-\frac{1}{1-\theta}\right){\rm d}\theta\\
&=-\log{\left(1+\theta\right)}-\log{\left(1-\theta\right)}+c=-\log{\left(1-\theta^2\right)}+ c,
\end{align*}
{we obtain the condition 
\begin{align*}
(\log{\xi(\theta)})^{\prime}&=\exp\Big[{\displaystyle \int\frac{2\theta}{1-\theta^2}{\rm d}\theta}\Big]\left\{\int \frac{2}{1-\theta^2}\exp\Big[{\int\frac{-2\theta}{1-\theta^2}{\rm d}\theta}\Big]{\rm d}\theta +c\right\}\\
&=\exp\Big\{-\log{\left(1-\theta^2\right)}\Big\}\left\{\int \frac{2}{1-\theta^2}\exp\Big\{\log{\left(1-\theta^2\right)}\Big\}{\rm d}\theta +c\right\}\\
&=\frac{1}{\left(1-\theta^2\right)}\left\{\int \frac{2}{1-\theta^2}\left(1-\theta^2\right){\rm d}\theta +c\right\}
=\frac{2\theta +c}{1-\theta^2}.
\end{align*}
}
Since, for any constant $c$, 
\begin{align*}
\frac{2\theta}{1-\theta^2}=\frac{1}{1-\theta}-\frac{1}{1+\theta}=-\left(\log{\left(1-\theta\right)}\right)^{\prime}-\left(\log{\left(1+\theta\right)}\right)^{\prime}
=\left[\log{\left\{\left(1-\theta^2\right)^{-1}\right\}}+c\right]^{\prime},
\end{align*}
any prior  
$ \xi \left(\theta\right)$ of the form $c(1-\theta^2)^{-1}$ 
  is  second-order robust.  The constant $c$ then should be chosen such that $\xi$ integrates up to one on $\Theta$ (recall that $\Theta$, in general,  depends on the specified components of~$\boldsymbol\theta$). Namely, the prior 
 \begin{equation}\label{xirobust}
 \xi \left(\theta\right)\coloneqq \Big[(1-\theta^2)\int_{\Theta} (1-\theta^2)^{-1}{\rm d}\theta\Big]  ^{-1},\quad \theta\in\Theta 
 \end{equation}
 is  second-order robust. 
}
\end{example}
\vspace{-2mm}

\begin{example}[ARMA$\left(p,q\right)$ case with $\theta=\psi_k$  { for some $1\leq k\leq q$}]
{\rm 
 Still for the ARMA$\left(p,q\right)$ case with spectral density (\ref{spectralARMA}), assume that $\theta$ is the MA  para\-meter~$\psi_k$,  for some~$1\leq k\leq q$ while all other parameters are specified. Then,  
\begin{align*}
&\Gamma\left(\theta\right)=\frac{1}{1-\theta^2},\quad K\left(\theta\right)=\frac{-6\theta}{\left(1-\theta^2\right)^2},\quad \text{and}\quad J\left(\theta\right)=\frac{4\theta}{\left(1-\theta^2\right)^2}
\end{align*}
(see p. 31 in  \cite{Taniguchi1991b}). 

Proceeding as in Example~\ref{ex:4-1}, we obtain, for  (\ref{eq:4-5}),  
\begin{align*}
(\log{\xi(\theta)})^{\prime\prime}-\frac{2\theta}{1-\theta^2}(\log{\xi(\theta)})^{\prime}-\frac{2}{1-\theta^2}=0, 
\end{align*}
hence 
$
\log{\xi(\theta)}=-\log{\left(1-\theta^2\right)}  +c
$.  
It follows that the same prior \eqref{xirobust} 
 as in Example \ref{ex:4-1}  
 is  second-order robust for $\psi_k$---a somewhat unexpected result.
}
\end{example}
 }
 
 \subsection{Second-order robustness for minimum contrast estimators}

In this section, we eva\-luate $R_Q(\theta)$ for a general class of minimum contrast estimators for Gaussian ARMA processes $\left\{X_t\right\}$ with spectral density $f_{\boldsymbol\theta}\left(\lambda\right)$. Let $\boldsymbol{X}_T\coloneqq \left(X_1,\ldots,X_T\right)^{\top}$, and let $A_T\left(\theta\right)$ and~$B_T\left(\theta\right)$ denote the $T\times T$ Toeplitz matrices with   $\left(m,l\right)$th elements   
 $$\int_{-\pi}^{\pi}\exp\left\{i\left(m-l\right)\lambda\right\}g_{\theta}\left(\lambda\right){\rm d}\lambda  \quad\text{and}\quad \int_{-\pi}^{\pi}\exp\left\{i\left(m-l\right)\lambda\right\}h_{\theta}\left(\lambda\right){\rm d}\lambda,\quad m,l = 1,\ldots,T,$$
  respectively. The spectral forms $g_{\theta}\left(\lambda\right)$ and $h_{\theta}\left(\lambda\right)$ satisfy
\begin{align}
g_{\theta}^{-1}\left(\lambda\right)h_{\theta}\left(\lambda\right)=\frac{1}{2}f_{\theta}^{-1}\left(\lambda\right)\frac{\partial}{\partial \theta}f_{\theta}\left(\lambda\right).
\end{align}
\cite{Taniguchi1987} introduced a general class $\mathcal{G}$ of minimum contrast estimates~$\widehat{\theta}_T^{M}$ defined as a value of~$\theta$ that satisfies the equation
\begin{align}
\frac{1}{T}\boldsymbol{X}_T^{\top}A_T^{-1}\left(\theta\right)B_T\left(\theta\right)A_T^{-1}\left(\theta\right)\boldsymbol{X}_T=b_T\left(\theta\right),
\end{align}
where $\displaystyle{b_T\left(\theta\right)\coloneqq \frac{1}{4\pi}\int^{\pi}_{-\pi}f_{\theta}^{-1}\left(\lambda\right)\frac{\partial}{\partial \theta}f_{\theta}\left(\lambda\right)d\lambda}$. That class $\mathcal{G}$ includes the MLE and  Whittle estimators of $\theta$. For $H_T\left(\theta\right)=A_T^{-1}\left(\theta\right)B_T\left(\theta\right)A_T^{-1}\left(\theta\right)$, let
\begin{align}
&\widetilde{Z}_1\left(\theta\right)=\frac{1}{\sqrt{T}}\left\{\boldsymbol{X}_T^{\top}H_T\left(\theta\right)\boldsymbol{X}_T-Tb_T\left(\theta\right)\right\}
\end{align}
and
\begin{align}
&\widetilde{Z}_2\left(\theta\right)=\frac{1}{\sqrt{T}}\left\{\boldsymbol{X}_T^{\top}\frac{\partial}{\partial\theta}H_T\left(\theta\right)\boldsymbol{X}_T-{\rm tr}\left\{E\left(\boldsymbol{X}_T\boldsymbol{X}_T^{\top}\right)\frac{\partial}{\partial\theta}H_T\left(\theta\right)\right\}\right\}.
\end{align}
Under appropriate regularity conditions,  \cite{Taniguchi1987} derived the stochastic expansion 
\begin{align}
\ \sqrt{T}\left(\widehat{\theta}_T^{M}\!\!-\theta\right) =\, & \Gamma^{-1}\left(\theta\right)\widetilde{Z}_1\left(\theta\right)\nonumber\\ \label{Eq410}
&\ \ +\frac{1}{\sqrt{T}\Gamma^{2}\left(\theta\right)}\left\{\widetilde{Z}_1\left(\theta\right)\widetilde{Z}_2\left(\theta\right)-\frac{\left\{3J\left(\theta\right)+K\left(\theta\right)\right\}}{2\Gamma\left(\theta\right)}\widetilde{Z}_1^2\left(\theta\right)\right\} +o_{\rm P}\left(\frac{1}{\sqrt{T}}\right)\\\nonumber
=\, & \Gamma^{-1}\left(\theta\right)\left[\widetilde{Z}_1\left(\theta\right)-E\left\{\widetilde{Z}_1\left(\theta\right)\right\}\right]\\\nonumber
&\ \ +\frac{1}{\sqrt{T}}\left\{\frac{\mu\left(\theta\right)}{\Gamma\left(\theta\right)}+\frac{\widetilde{Z}_1\left(\theta\right)\widetilde{Z}_2\left(\theta\right)}{\Gamma^{2}\left(\theta\right)}-\frac{\left\{3J\left(\theta\right)+K\left(\theta\right)\right\}\widetilde{Z}_1^2\left(\theta\right)}{2\Gamma^3\left(\theta\right)}\right\}+o_{\rm P}\left(\frac{1}{\sqrt{T}}\right)\\\nonumber
=&\,\Gamma^{-1}\left\{\widetilde{Z}_1-E\left(\widetilde{Z}_1\right)\right\}+\frac{1}{\sqrt{T}}\widetilde{Q}+o_{\rm P}\left(\frac{1}{\sqrt{T}}\right) ,\mbox{ say},
\end{align}
for $\widehat{\theta}_T^{M}$, where $E\left(\widetilde{Z}_1\right)=\frac{1}{\sqrt{T}}\mu\left(\theta\right)+o\left(\frac{1}{\sqrt{T}}\right)$. We then   obtain
\begin{align}
R_{Q}(\theta)=\left(\frac{\mu}{\Gamma}\right)^{\prime}+\left(\frac{-J-K}{2\Gamma^2}\right)^{\prime}.
\end{align}

We now proceed with the evaluation of $\Gamma\left(\theta\right)$, $K\left(\theta\right)$, and $J\left(\theta\right)$ for various models with rational spectral densities. Furthermore, we also derive the quantities~$\mu_{\left(\cdot\right)}\left(\theta\right)$ for the maximum likelihood estimators~$\widehat{\theta}_{\text{\,\rm ML}}$ and the Whittle estimators $\widehat{\theta}_{\text{\,\rm Whittle}}$. We then obtain the second-order robustness $R_{Q}\big(\widehat{\theta}_{\left(\cdot\right)}\big)$ of these estimators and compare their robustness  via  the ratio~$r\coloneqq  {R_{Q}\big(\widehat{\theta}_{\text{\,\rm Whittle}}\big)}/{R_{Q}\big(\widehat{\theta}_{\text{\,\rm ML}}\big)}$ or the difference~$\tilde{r}\coloneqq  {R_{Q}\big(\widehat{\theta}_{\text{\,\rm Whittle}}\big)} - {R_{Q}\big(\widehat{\theta}_{\text{\,\rm ML}}\big)}$. 

\begin{example}[AR$\left(1\right)$ model]

{\rm 
For the AR$\left(1\right)$ process 
 $
X_t-\rho X_{t-1}=\varepsilon_t 
$ \linebreak 
with~$\vert\rho\vert <~\!1$ (hence $\Theta = (-1,\, 1)$) and $\varepsilon_t\stackrel{i.i.d.}{\sim}N\left(0,\sigma^2\right)$,  \cite{Taniguchi1991b} pp. 30--33 obtains 
\begin{align*}
\Gamma\left(\rho\right)=\frac{1}{1-\rho^2},\quad K\left(\rho\right)&=\frac{6\rho}{\left(1-\rho^2\right)^2},\quad J\left(\rho\right)=\frac{-2\rho}{\left(1-\rho^2\right)^2},\\
\mu_{\text{\,\rm ML}}\left(\rho\right)=0,\quad & \text{and}\quad \mu_{\text{\,\rm Whittle}}\left(\rho\right)=\frac{-\rho}{1-\rho^2}
\end{align*}
($\sigma^2$ plays no role here and can safely be omitted in the analysis). Therefore, 
\begin{align*}
&R_{Q}\left(\widehat{\rho}_{\text{\,\rm ML}}\right)=0+\left(-2\rho\right)^{\prime}=-2\quad\text{and}\quad 
R_{Q}\left(\widehat{\rho}_{\text{\,\rm Whittle}}\right)=\left(-\rho\right)^{\prime}+\left(-2\rho\right)^{\prime}=-1-2,
\end{align*}
hence
 $
r\coloneqq 
{R_{Q}\left(\widehat{\rho}_{\text{\,\rm Whittle}}\right)}/{R_{Q}\left(\widehat{\rho}_{\text{\,\rm ML}}\right)}=\frac{-1}{-2}+1>1.
$
It follows that 
the maximum likelihood estimator~$\widehat{\rho}_{\text{\,\rm ML}}$ is uniformly more robust than the Whittle estimator~$\widehat{\rho}_{\text{\,\rm Whittle}}$ for AR$\left(1\right)$ processes. This result is plausible, since the Whittle estimator is known to have a large bias.
}
\end{example}

\begin{example}[MA$\left(1\right)$ model]
{\rm 
For the MA$\left(1\right)$ process
 $
X_t=\varepsilon_t-\psi\varepsilon_{t-1}
$ 
with~$\vert\psi\vert <~\!1$  (hence $\Theta = (-1,\, 1)$) and~$\varepsilon_t\stackrel{i.i.d.}{\sim}N\left(0,\sigma^2\right)$, we obtain
\begin{align*}
\Gamma\left(\psi\right)=\frac{1}{1-\psi^2},\quad K\left(\psi\right)&=\frac{-6\psi}{\left(1-\psi^2\right)^2},\quad J\left(\psi\right)=\frac{4\psi}{\left(1-\psi^2\right)^2},\\
\mu_{\text{\,\rm ML}}\left(\psi\right)=0,&\quad\mu_{\text{\,\rm Whittle}}\left(\psi\right)=\frac{-\psi\left(1+\psi^2\right)}{\left(1-\psi^2\right)^2}.
\end{align*}
Then, we see that
\begin{align*}
&R_{Q}\left(\widehat{\psi}_{\text{\,\rm ML}}\right)=0+\left(\psi\right)^{\prime}=1
\end{align*}
and
\begin{align*}
R_{Q}\left(\widehat{\psi}_{\text{\,\rm Whittle}}\right)=\left(\frac{-\psi-\psi^3}{1-\psi^2}\right)^{\prime}+\left(\psi\right)^{\prime}=\frac{\psi^4-4\psi^2-1}{\left(1-\psi^2\right)^2}+1,
\end{align*}
and, therefore,
\begin{align*}
r\coloneqq \frac{R_{Q}\left(\widehat{\psi}_{\text{\,\rm Whittle}}\right)}{R_{Q}\left(\widehat{\psi}_{\text{\,\rm ML}}\right)}=\frac{\psi^4-4\psi^2-1}{\left(1-\psi^2\right)^2}+1<1.
\end{align*}
The Whittle estimator $\widehat{\rho}_{\text{\,\rm Whittle}}$, thus,  is always more robust here than the maxi\-mum likelihood estimator~$\widehat{\rho}_{\text{\,\rm ML}}$---reversing the  conclusion previously made in Example~4.3 for  AR$\left(1\right)$ models. 
}
\end{example}

\begin{example}[ARMA$\left(1,1\right)$ model, unspecified $\theta = \sigma^2$, specified  $\rho$ and $\psi$]

{\rm 
For the ARMA$\left(1,1\right)$ process
 $
X_t-\rho X_{t-1}=\varepsilon_t-\psi\varepsilon_{t-1}
$ 
with $\varepsilon_t\stackrel{i.i.d.}{\sim}N\left(0,\sigma^2\right)$, where~$\theta = \sigma^2$ is unspecified (with $\Theta=(0,\infty)$) and $(\rho, \psi)\in (-1,\, 1)^2$ are specified, we obtain
\begin{align*}
\Gamma\left(\sigma^2\right)=\frac{1}{2\sigma^4},&\quad K\left(\sigma^2\right)=\frac{1}{\sigma^6},\quad J\left(\sigma^2\right)=-\frac{1}{\sigma^6},\\
\mu_{\text{\,\rm ML}}\left(\sigma^2\right)=0,\quad&\text{and}\quad\mu_{\text{\,\rm Whittle}}\left(\sigma^2\right)=\frac{\left(\rho-\psi\right)^2}{\sigma^2\left(1-\rho^2\right)\left(1-\psi^2\right)}.
\end{align*}
Then, we see that
\begin{align*}
&R_{Q}\left(\widehat{\sigma}^2_{\text{\,\rm ML}}\right)=0+0=0,\\
&R_{Q}\left(\widehat{\sigma}^2_{\text{\,\rm Whittle}}\right)=\left\{\frac{2\sigma^2\left(\rho-\psi\right)^2}{\left(1-\rho^2\right)\left(1-\psi^2\right)}\right\}^{\prime}+0=\frac{2\left(\rho-\psi\right)^2}{\left(1-\rho^2\right)\left(1-\psi^2\right)},
\end{align*}
and, therefore,
\begin{align*}
\widetilde{r}\coloneqq R_{Q}\left(\widehat{\sigma}^2_{\text{\,\rm Whittle}}\right)-R_{Q}\left(\widehat{\sigma}^2_{\text{\,\rm ML}}\right)=\frac{2\left(\rho-\psi\right)^2}{\left(1-\rho^2\right)\left(1-\psi^2\right)}\ge 0.
\end{align*}
That is, the maximum likelihood estimator $\widehat{\rho}_{\text{\,\rm ML}}$ is always more robust than the Whittle estimator~$\widehat{\rho}_{\text{\,\rm Whittle}}$.  
}
\end{example}

\begin{figure}[b!]
  \begin{center}
\includegraphics[width=3in]{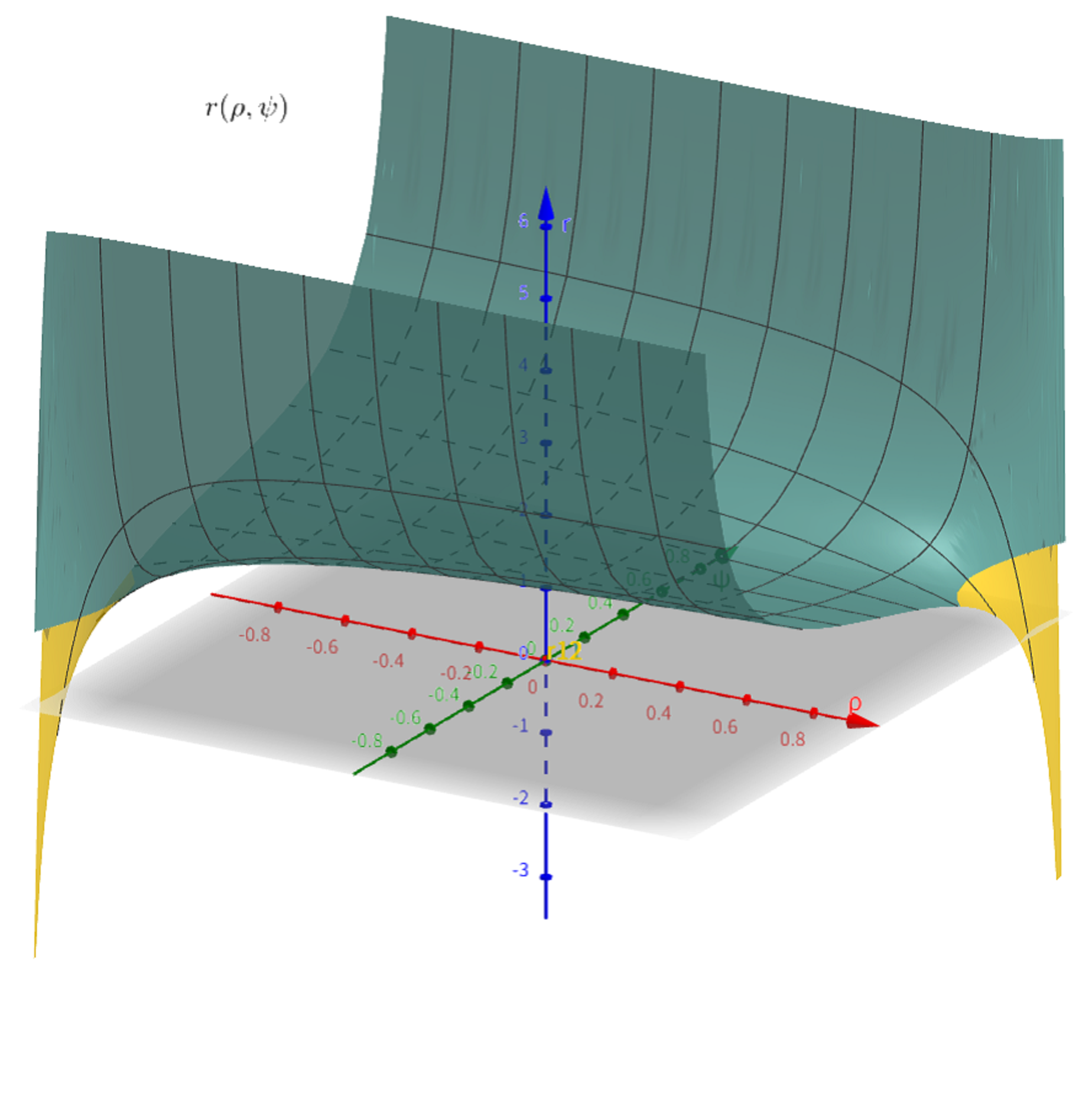}
\vspace{-15mm}  \end{center}
\caption{The graph of $\left(\rho,\psi\right)\mapsto r\left(\rho,\psi\right)$ for ARMA$\left(1,1\right)$ with $\rho$ unspecified, $\psi$ and  $\sigma^2$ specified (Example 4.6).  
 The parts of the graph 
  where the maximum likelihood estimator~$\widehat{\rho}_{\text{\,\rm ML}}$ is more (respectively, less) robust than the Whittle estimator $\widehat{\rho}_{\text{\,\rm Whittle}}$  are colored in green (respectively, in yellow).
}
\label{fig:10}
\end{figure}

\begin{example}[ARMA$\left(1,1\right)$ model, unspecified $\theta = \rho$, specified $\psi$ and  $\sigma^2$]
{\rm 
For the same ARMA$\left(1,1\right)$ process  
$
X_t-\rho X_{t-1}=\varepsilon_t-\psi\varepsilon_{t-1}
$ 
with $\varepsilon_t\stackrel{i.i.d.}{\sim}N\left(0,\sigma^2\right)$, where $\theta=\rho$ is unspecified, and $\psi$, $\sigma^2$ are specified, we obtain 
(note that, in this model, $\boldsymbol\Theta = (-1,\, 1)^2$ and~$\Theta = (-1,\, 1)$ irrespective of the value of~$\psi$)
\begin{align*}
&\ \ \Gamma\left(\rho\right)=\frac{1}{1-\rho^2},\quad K\left(\rho\right)=\frac{6\rho}{\left(1-\rho^2\right)^2},\quad J\left(\rho\right)=\frac{-2\rho}{\left(1-\rho^2\right)^2},\\
&\mu_{\text{\,\rm ML}}\left(\rho\right)=0,\quad\text{and}\quad \mu_{\text{\,\rm Whittle}}\left(\rho\right)=\frac{-\left(\rho-\psi\right)\left(1-2\rho\psi+\psi^2\right)}{\left(1-\rho^2\right)\left(1-\rho\psi\right)\left(1-\psi^2\right)}.
\end{align*}
Therefore, 
\begin{align*}
R_{Q}\left(\widehat{\rho}_{\text{\,\rm ML}}\right)&=0+\left(-2\rho\right)^{\prime}=-2,\\
R_{Q}\left(\widehat{\rho}_{\text{\,\rm Whittle}}\right)&=\left\{\frac{-\left(\rho-\psi\right)\left(1-2\rho\psi+\psi^2\right)}{\left(1-\rho\psi\right)\left(1-\psi^2\right)}\right\}^{\prime}+\left(-2\rho\right)^{\prime}\\
&=-\frac{\left(1+2\psi^2-\psi^4-4\rho\psi+2\rho^2\psi^2\right)}{\left(1-\psi^2\right)\left(1-\rho\psi\right)^2}-2,
\end{align*}
hence 
\begin{align*}
r\coloneqq \frac{R_{Q}\left(\widehat{\rho}_{\text{\,\rm Whittle}}\right)}{R_{Q}\left(\widehat{\rho}_{\text{\,\rm ML}}\right)}=\frac{\left(1+2\psi^2-\psi^4-4\rho\psi+2\rho^2\psi^2\right)}{2\left(1-\psi^2\right)\left(1-\rho\psi\right)^2}+1.
\end{align*}

\begin{figure}[b!]
\begin{center}
\includegraphics[width=2.3in]{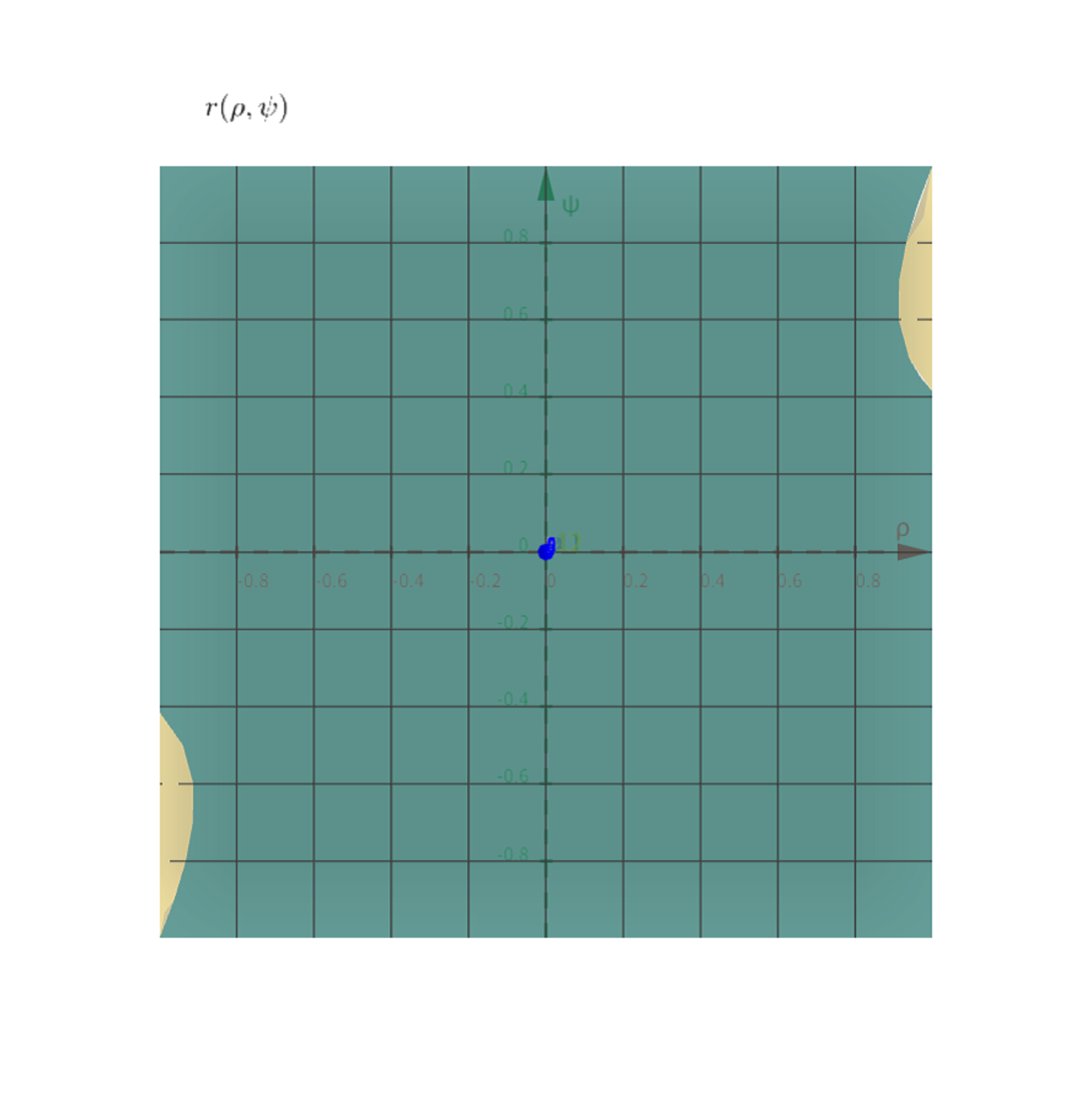}
\includegraphics[width=2.3in]{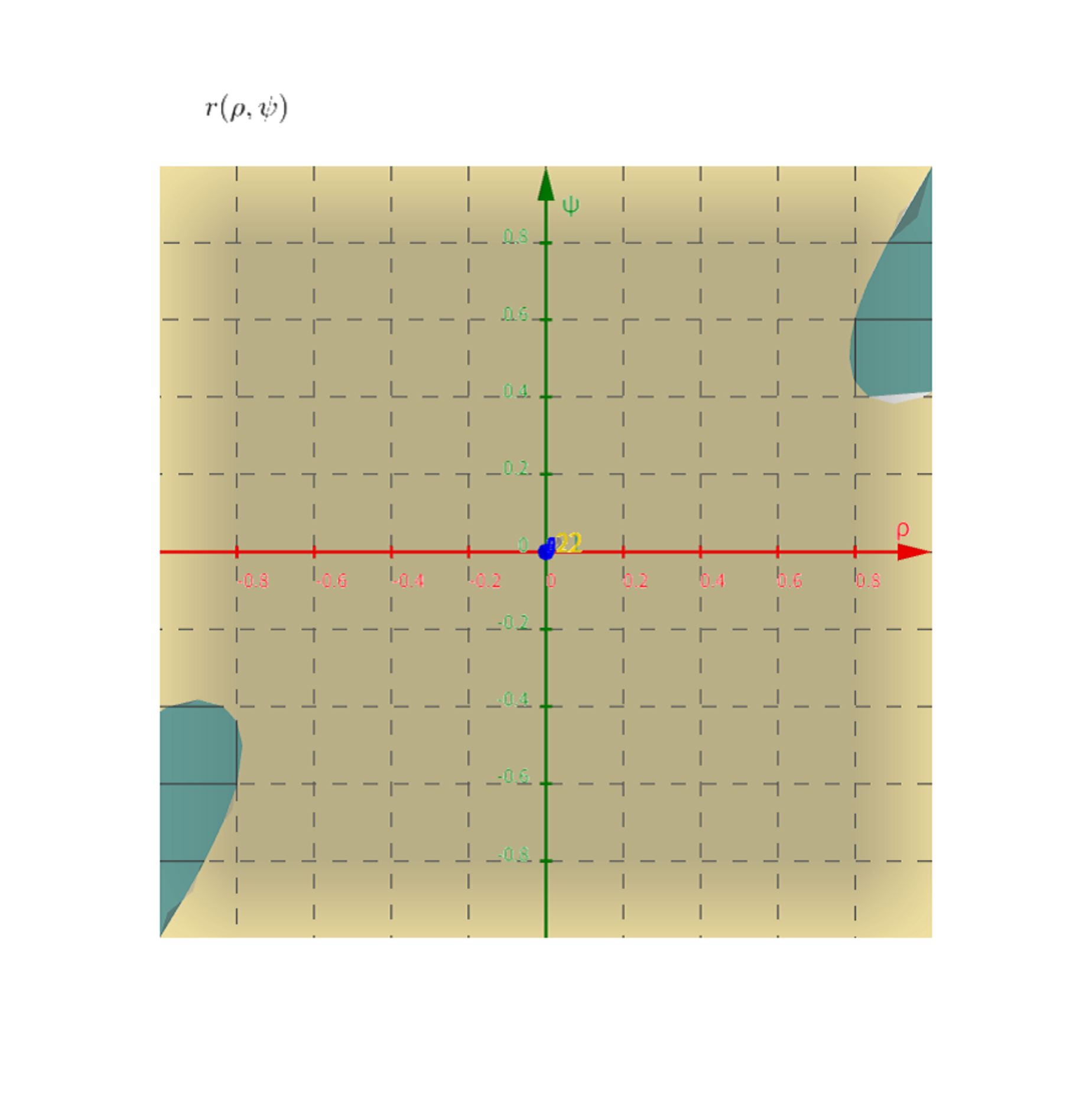}
\end{center}
\caption{The top-down views (from the direction of ``the North Pole'') of the graph of $r\left(\rho,\psi\right)$ for ARMA$\left(1,1\right)$  with~$\rho$ unspecified, $\psi$ and  $\sigma^2$ specified (Example 4.6, left-hand panel) and  with~$\psi$ unspecified, $\rho$ and  $\sigma^2$ specified (Example 4.7, right-hand panel), respectively.  The~$(\rho,\psi)$ values at which the maximum likelihood estimator~$\widehat{\rho}_{\text{\,\rm ML}}$ is more robust than the Whittle estimator~$\widehat{\rho}_{\text{\,\rm Whittle}}$ are  in green, those at which  the Whittle estimator $\widehat{\rho}_{\text{\,\rm Whittle}}$ is more robust than the maximum likelihood estimator $\widehat{\rho}_{\text{\,\rm ML}}$ are in yellow. }
\label{fig:120}
\end{figure}

The comparison between $\widehat{\rho}_{\text{\,\rm ML}}$ and $\widehat{\rho}_{\text{\,\rm Whittle}}$, thus, depends on both $\rho$ and $\psi$, but not on~$\sigma^2$.  The graph of~$(\rho,\psi)\mapsto r\left(\rho,\psi\right)$ for ARMA$\left(1,1\right)$ models, where $\rho$ is unspecified, and $\psi$, $\sigma^2$ are specified, is shown in Figure \ref{fig:10}. The 
 range in which the maximum likelihood esti\-mator~$\widehat{\rho}_{\text{\,\rm ML}}$ is more robust than the Whittle estimator~$\widehat{\rho}_{\text{\,\rm Whittle}}$ (the $(\rho,\psi)$ values such that $r\left(\rho,\psi\right)>1$) and the range in which the Whittle estimator~$\widehat{\rho}_{\text{\,\rm Whittle}}$ is more robust than the maximum likelihood estimator $\widehat{\rho}_{\text{\,\rm ML}}$ (i.e., $r\left(\rho,\psi\right)<1$)   are colored in green and yellow, respectively. 
  The top-down view (from the direction of ``the North Pole'') of the graph of $r\left(\rho,\psi\right)$ is provided in Figure~\ref{fig:120} (left-hand panel).  
 It is seen that $r\left(\rho,\psi\right)$ is predominantly above 1 and tends to $+\infty$ as $\left|\psi\right|\to 1$ (except around $\left(\rho,\psi\right)=\pm\left(1,1\right)$),   
  while~$r\left(\rho,\psi\right)$ tends to $-\infty$ as~$\rho\times\psi\to 1$. The innovation variance $\sigma^2$ plays no role. 
}
\end{example}
\begin{figure}[b!]
  \begin{center}
\includegraphics[width=3in]{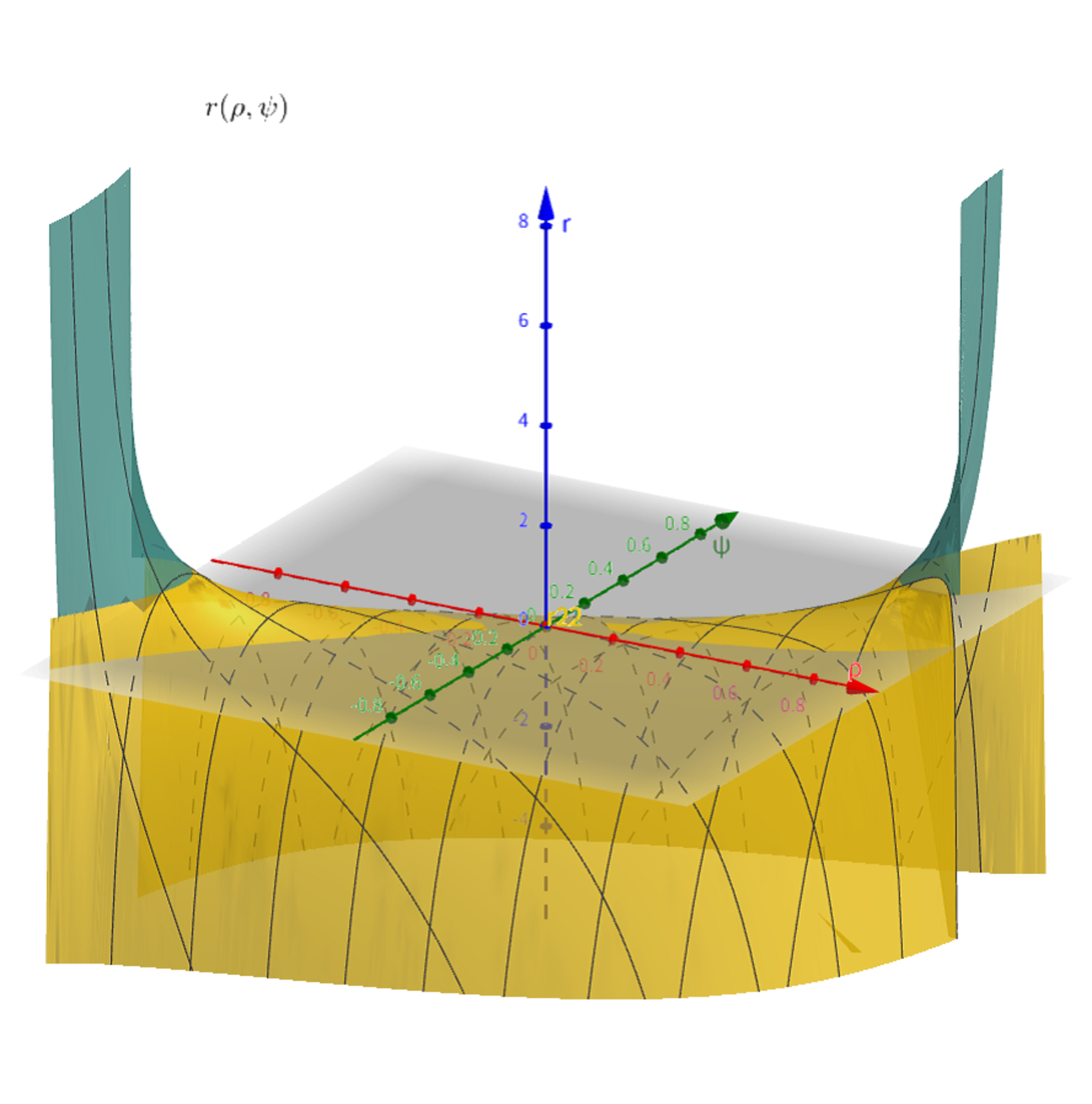}
  \end{center}
\caption{The graph of $\left(\rho,\psi\right)\mapsto r\left(\rho,\psi\right)$ for ARMA$\left(1,1\right)$ with $\psi$ unspecified, $\rho$ and  $\sigma^2$ specified (Example 4.7). The parts of the graph   where the maximum likelihood estimator~$\widehat{\rho}_{\text{\,\rm ML}}$ is more (respectively, less) robust than the Whittle estimator $\widehat{\rho}_{\text{\,\rm Whittle}}$  are colored in green (respectively, in yellow).  
}
\label{fig:20}
\end{figure}

\begin{example}[ARMA$\left(1,1\right)$ model, $\psi$ unspecified, $\rho$ and $\sigma^2$specified]
\rm 
Still for the ARMA$\left(1,1\right)$ process
 $
X_t-\rho X_{t-1}=\varepsilon_t-\psi\varepsilon_{t-1}
$ 
with $\varepsilon_t\stackrel{i.i.d.}{\sim}N\left(0,\sigma^2\right)$, now with $\psi$  unspe\-cified and $\rho$, $\sigma^2$ specified, we obtain
\begin{align*}
\Gamma\left(\psi\right)=\frac{1}{1-\psi^2},\quad  K\left(\psi\right)&=\frac{-6\psi}{\left(1-\psi^2\right)^2},\quad J\left(\psi\right)=\frac{4\psi}{\left(1-\psi^2\right)^2},\\
\mu_{\text{\,\rm ML}}\left(\psi\right)=0,\  \ \text{and}\ \ \mu_{\text{\,\rm Whittle}}\left(\psi\right)=&\frac{-\left(\psi-\rho\right)\left(1+\psi^2-2\psi\rho-\rho^2+3\psi^2\rho^2-2\psi^3\rho\right)}{\left(1-\psi^2\right)^2\left(1-\psi\rho\right)\left(1-\rho^2\right)}.
\end{align*}
Then, 
\begin{align*}
&R_{Q}\left(\widehat{\psi}_{\text{\,\rm ML}}\right)=0+\left(\psi\right)^{\prime}=1,\\
&R_{Q}\left(\widehat{\psi}_{\text{\,\rm Whittle}}\right)=\left\{\frac{-\left(\psi-\rho\right)\left(1\!+\! \psi^2-2\psi\rho-\rho^2\!+\! 3\psi^2\rho^2-2\psi^3\rho\right)}{\left(1-\psi^2\right)\left(1-\psi\rho\right)\left(1-\rho^2\right)}\right\}^{\prime}\!+\! \left(\psi\right)^{\prime}\\
&=\frac{2 \rho^2\psi^6\!\!-\!4 \rho\psi^5\!\!-\!3\rho^4\psi^4\!\!-\!4\rho^2\psi^4\!+\! \psi^4\!+\! 12 \rho^3\psi^3\!+\! 12 \rho\psi^3\!\!-\!22\rho^2\psi^2\!\!-\!4\psi^2\!+\! 4\rho^3\psi\!+\! 8\rho\psi-\rho^4\!\!-\!1}{\left(1\!\!-\!\psi^2\right)^2\left(1\!\!-\!\psi\rho\right)^2\left(1\!\!-\!\rho^2\right)}\!+\! 1,
\end{align*}
and, therefore,
\begin{align*}
&r\coloneqq 
{R_{Q}\left(\widehat{\psi}_{\text{\,\rm Whittle}}\right)}/{R_{Q}\left(\widehat{\psi}_{\text{\,\rm ML}}\right)}\\
&=\frac{2 \rho^2\psi^6\!-\! 4 \rho\psi^5\!-\! 3\rho^4\psi^4\!-\! 4\rho^2\psi^4+\psi^4\!+\! 12 \rho^3\psi^3\!+\! 12 \rho\psi^3\!-\! 22\rho^2\psi^2\!-\! 4\psi^2\!+\! 4\rho^3\psi\!+\! 8\rho\psi\!-\! \rho^4\!-\! 1}{\left(1\!-\! \psi^2\right)^2\left(1\!-\! \psi\rho\right)^2\left(1\!-\! \rho^2\right)}\!+\! 1.
\end{align*}

The resulting graph of $\left(\rho,\psi\right)\mapsto r\left(\rho,\psi\right)$ 
  is plotted in Figure \ref{fig:20}. The 
 range in which the maximum likelihood esti\-mator~$\widehat{\rho}_{\text{\,\rm ML}}$ is more robust than the Whittle estimator $\widehat{\rho}_{\text{\,\rm Whittle}}$ (the~$(\rho,\psi)$ values such that~$r\left(\rho,\psi\right)>1$) and the range in which the Whittle estimator $\widehat{\rho}_{\text{\,\rm Whittle}}$ is more robust than the maximum likelihood estimator $\widehat{\rho}_{\text{\,\rm ML}}$ (i.e., $r\left(\rho,\psi\right)<1$) are colored in green and yellow, respectively. The top-down view (from the direction of ``the North Pole'') of the graph of $r\left(\rho,\psi\right)$ is provided in Figure~\ref{fig:120} (right-hand panel).  
\end{example}


\begin{funding}
 The first author was supported by JSPS KAKENHI Grant-in-Aid for Scientific Research (C) 23k11004 and Nanzan University Pache Research Subsidy I-A-2 for the 2026 academic year (J. Hirukawa, Nanzan Univ.)  The second author was supported by JSPS KAKENHI Grant-in-Aid for Scientific Research (C) 26K14735 (M. Taniguchi, Waseda Univ.) The third author acknowledges support from  the Czech Science Foundation grant GA24-10078S and the Fonds Thelam of the Fondation Roi Baudouin.
\end{funding}



%
\bibliographystyle{imsart-nameyear.bst}
\bibliography{Convolution.bib}      


\begin{appendix}
\section*{}

\begin{proof}[Proof of Lemma \ref{lem:3-2}] 

Taking (\ref{eq:a2-2}) into account, we obtain
\begin{align*}
\exp{\left(it\, S_{\theta,T} +\Lambda_T\right)}=&\exp{\left\{it\, a_1W_{1,T}\left(\theta\right)+h\Gamma^{1/2}\left(\theta\right)W_{1,T}\left(\theta\right)-{h^2}\Gamma\left(\theta\right)/2\right\}}\\
&\  \ \ \ \ \ \ \ \  \times \left[1+c_T^{-1}q_{\theta ,T}\left(W_{1,T}\left(\theta\right),W_{2,T}\left(\theta\right)\right)\right]+O_{\rm P}\left(c_T^{-2}\right)
\end{align*}
where (with $Q_{\theta,T}$ as  in the expansion of $S_{\theta,T}$ for  $\{\widehat{\theta}_T\vert\, T\in{\mathbb N}\} \in \mathcal{D}_\theta$, see \eqref{eq:3-10}) 
\begin{align*}
q_{\theta ,T}\!\left(w_1,w_2\right)\!\coloneqq\! it\, Q_{\theta ,T}\left(w_1,w_2\right) 
&
+\frac{h^2}{2}\left\{J\left(\theta\right)\Gamma^{-1/2}\left(\theta\right)w_1+\sigma_2w_2\right\}
\! -\! \frac{h^3}{6}\left\{3J\left(\theta\right)+K\left(\theta\right)\right\} . 
\end{align*}

In view of Assumption \ref{assum:B1-2}, we can evaluate the asymptotic third-order cumulants\linebreak of~$\boldsymbol{W}_T\left(\theta\right)=\left(W_{1,T}\left(\theta\right),W_{2,T}\left(\theta\right)\right)^{\top}$: namely, 
\begin{align*}
{\rm cum}_{\theta}\left(W_{i,T}\left(\theta\right),W_{j,T}\left(\theta\right),W_{k,T}\left(\theta\right)\right)\coloneqq c_T^{-1}C_{ijk}+O\left(c_T^{-2}\right)\quad i,j,k=1,2. 
\end{align*}
In particular $C_{111}=K\left(\theta\right)\Gamma^{-3/2}\left(\theta\right)$. We also obtain the second-order Edgeworth expansion 
\begin{align*}
&f_{\boldsymbol{W}_T}\left(w_1,w_2\right)=\phi\left(w_1\right)\phi\left(w_2\right)\left\{1+c_T^{-1}r\left(w_1,w_2\right)\right\}+O\left(c_T^{-2}\right)\quad\text{as $T\to\infty$}
\end{align*}
of the density $f_{\boldsymbol{W}_T}$ of $\boldsymbol{W}_T$,  
where $\phi\left(w\right)\coloneqq \left(2\pi\right)^{-1/2}
\exp{\left\{-{w^2}/{2}\right\}}$ and
\begin{align*}
r\left(w_1,w_2\right)\coloneqq &\frac{1}{6}\left\{C_{111}\left(w_1^3-3w_1\right)+3\,C_{112}w_2\left(w_1^2-1\right)\right.\\
&\left. \qquad\qquad\qquad  +\,3\,C_{122}w_1\left(w_2^2-1\right)+C_{222}\left(w_2^3-3w_2\right)\right\}.
\end{align*}

We have 
\begin{align*}
\psi_{S_{\theta,T}}\left(t,h\right)&=\int^{\infty}_{-\infty}\int^{\infty}_{-\infty}\phi\left(w_1\right)\phi\left(w_2\right)\exp{\left\{it\, a_1w_{1}+h\Gamma^{1/2}\left(\theta\right)w_{1}-\frac{h^2}{2}\Gamma\left(\theta\right)\right\}}\\
&\qquad\qquad\qquad\times\left(1+c_T^{-1}q_{\theta ,T}\left(w_{1},w_{2}\right)+c_T^{-1}r\left(w_1,w_2\right)\right){\rm d}w_1{\rm d}w_2+O\left(c_T^{-2}\right)\\
%
&=\int^{\infty}_{-\infty}\phi\left(w_1\right)\exp{\left\{it\,a_1w_{1}+h\Gamma^{1/2}\left(\theta\right)w_{1}-\frac{h^2}{2}\Gamma\left(\theta\right)\right\}}\\
&\qquad\qquad\qquad\times \left(1+c_T^{-1}\widetilde{q}_{\theta,T}\left(w_{1}\right)+c_T^{-1}\widetilde{r}\left(w_1\right)\right){\rm d}w_1+O\left(c_T^{-2}\right)\\
&=\exp{\Big\{\frac{\left(it\right)^2a_1^2}{2}+it\,a_1\Gamma^{1/2}\left(\theta\right)h\Big\}}\int^{\infty}_{-\infty}\left(2\pi\right)^{-1/2}\exp{\left\{-{\left(w_1-\mu\right)^{2}}/{2}\right\}}\\
&\qquad\qquad\qquad\qquad\times\left(1+c_T^{-1}\widetilde{q}_{\theta,T}\left(w_{1}\right)+c_T^{-1}\widetilde{r}\left(w_1\right)\right){\rm d}w_1+O\left(c_T^{-2}\right)
\end{align*}
where
\begin{align}\label{rtilde}
&\widetilde{r}\left(w_1\right)\coloneqq \int^{\infty}_{-\infty}\!\!\phi\left(w_2\right)r\left(w_1,w_2\right){\rm d}w_2,\quad\widetilde{q}_{\theta,T}\left(w_1\right)
\coloneqq \int^{\infty}_{-\infty}\!\!\phi\left(w_2\right)q\left(w_1,w_2\right){\rm d}w_2,
\end{align} 
and $\mu\coloneqq it\,a_1+h\Gamma^{1/2}\left(\theta\right)$. 
Since
\begin{align*}
\int^{\infty}_{-\infty}&\left(2\pi\right)^{-1/2}\exp \left\{-{\left(w_1-\mu\right)^{2}}/{2}\right\}\widetilde{q}_{\theta,T}\left(w_{1}\right){\rm d}w_1\\
&\quad\quad =it \int^{\infty}_{-\infty}\int^{\infty}_{-\infty}\left(2\pi\right)^{-1/2}\exp{\left\{- {\left(w_1-\mu\right)^{2}}/{2}\right\}}\phi\left(w_2\right)Q_{\theta,T}\left(w_1,w_2\right){\rm d}w_1{\rm d}w_2\\
&\qquad\qquad\qquad\quad +\frac{h^2J\left(\theta\right)\Gamma^{-1/2}\left(\theta\right)}{2}\mu-\frac{h^3}{6}\left\{3J\left(\theta\right)+K\left(\theta\right)\right\}\\
&\quad\quad = it \int^{\infty}_{-\infty}\left(2\pi\right)^{-1/2}\exp{\left\{- {\left(w_1-\mu\right)^{2}}/{2}\right\}}\overline{Q}_{\theta,T}\left(w_1\right){\rm d}w_1\\
&\qquad\qquad\qquad\quad +it\,a_1\frac{h^2J\left(\theta\right)\Gamma^{-1/2}\left(\theta\right)}{2}-\frac{h^3K\left(\theta\right)}{6}
\eqqcolon it\,a_2\left(\mu\right)+b_1\vspace{-3mm}
\end{align*}
with
\begin{align*}
a_2\left(\mu\right)&\coloneqq \int^{\infty}_{-\infty}\left(2\pi\right)^{-1/2}\exp{\left\{- {\left(w_1-\mu\right)^{2}}/{2}\right\}}\overline{Q}_{\theta,T}\left(w_1\right){\rm d}w_1\\
&=\int^{\infty}_{-\infty}\phi\left(w_1\right)\overline{Q}_{\theta,T}\left(w_1+\mu\right){\rm d}w_1={\rm E}_{{\theta,T}}\left\{Q_{\theta,T}\left(W_{1,T}\left(\theta\right)+\mu,W_{2,T}\left(\theta\right)\right)\right\}+o\left(1\right),
\end{align*}
$$\overline{Q}_{\theta,T}\left(w_1\right) \coloneqq \int^{\infty}_{-\infty}\phi\left(w_2\right)Q_{\theta,T}\left(w_1,w_2\right){\rm d}w_2,\ \  \text{and}\ \  b_1\coloneqq it\,a_1\frac{h^2J\left(\theta\right)\Gamma^{-1/2}\left(\theta\right)}{2}-\frac{h^3K\left(\theta\right)}{6}, 
$$
letting  
\begin{align*}
\int^{\infty}_{-\infty}\left(2\pi\right)^{-1/2}\exp{\left\{-{\left(w_1-\mu\right)^{2}}/{2}\right\}}\widetilde{r}\left(w_{1}\right){\rm d}w_1
&=\frac{K\left(\theta\right)\Gamma^{-3/2}\left(\theta\right)}{6}\mu^3
\eqqcolon  b_2
\end{align*}
with $\tilde r$ defined in \eqref{rtilde} and\vspace{-1mm} 
\[
b_3\coloneqq b_1 + b_2 = it\,a_1\frac{h^2J\left(\theta\right)\Gamma^{-1/2}\left(\theta\right)}{2}-\frac{h^3K\left(\theta\right)}{6}+\frac{K\left(\theta\right)\Gamma^{-3/2}\left(\theta\right)}{6}\mu^3 ,
\]
 it follows that\vspace{-1mm} 
\begin{align*}
\psi_{S_{\theta,T}}\left(t,h\right)=\exp{\left\{\frac{\left(it\right)^2a_1^2}{2}+it\,a_1\Gamma^{1/2}\left(\theta\right)h\right\}}\left[1+c_T^{-1}\left\{it\,a_2\left(\mu\right)+b_3\right\}\right]+O\left(c_T^{-2}\right)
\end{align*}
while\vspace{-3mm} 
\begin{align*}
\psi_{S_{\theta_T, T}}\left(t,h\right)&=\exp {(-it\, h)}\psi_{S_{\theta,t}}\left(t,h\right)\\ 
&=\exp{\left\{\frac{\left(it\right)^2a_1^2}{2}+it\left\{a_1\Gamma^{1/2}\left(\theta\right)-1\right\}h\right\}}\\
&\qquad\times \left[1+c_T^{-1}\left\{it\,a_2\left(\mu\right)
 +it\,a_1\frac{h^2J\left(\theta\right)\Gamma^{-1/2}\left(\theta\right)}{2}-\frac{h^3K\left(\theta\right)}{6}\right.\right.  \\ 
 &\qquad  \quad\quad\qquad \ 
\left.\left. +\frac{K\left(\theta\right)\Gamma^{-3/2}\left(\theta\right)}{6}\mu^3\right\}\right]+O\left(c_T^{-2}\right)\quad\text{as $T\to\infty$.}
\end{align*}\vspace{-11.5mm}

\end{proof}
\end{appendix}

\end{document}